\documentclass[11pt]{amsart}
\usepackage[usenames,dvipsnames]{color}
\usepackage[table]{xcolor}
\usepackage{slashbox}
\usepackage{multirow}
\usepackage{amsfonts,epsfig}
\usepackage{latexsym}
\usepackage{amssymb}
\usepackage{amsmath}
\usepackage{amsthm}
\usepackage{graphics}
\usepackage[all]{xy}
\usepackage[T2A]{fontenc}
\usepackage{multirow}
\usepackage{array, booktabs, ctable}
\usepackage{enumerate}
\usepackage{booktabs}
\usepackage{longtable}
\usepackage{color}
\usepackage{multirow}
\usepackage{hyperref}
\usepackage{xcolor,colortbl}

\renewcommand{\arraystretch}{1.3}

\newtheorem{theorem}{Theorem}
\newtheorem{corollary}[theorem]{Corollary}
\newtheorem{lemma}[theorem]{Lemma}
\newtheorem{proposition}[theorem]{Proposition}

\newtheorem*{question}{Question}
\newtheorem*{conjecture}{Conjecture}

\theoremstyle{definition}\newtheorem{remark}{Remark}

\newcommand{\cC}{{\mathcal{C}}}
\DeclareMathOperator{\cm}{CM} 
\DeclareMathOperator{\tors}{{tors}}

\newcommand{\Q}{\mathbb Q}
\newcommand{\Qbar}{{\overline{\mathbb Q}}} 
\newcommand{\Z}{\mathbb Z}

\newcommand{\SL}{\operatorname{SL}}

\newcommand{\Gal}{\operatorname{Gal}}
\newcommand{\Aut}{\operatorname{Aut}}

\newcommand{\GL}{\operatorname{GL}}

\newenvironment{romanenum}{\hfill \begin{enumerate} }{\end{enumerate}}
\newenvironment{alphenum}{\hfill \begin{enumerate} }{\end{enumerate}}

\newenvironment{alphenumCM}{\hfill \begin{enumerate} }{\end{enumerate}}

\definecolor{light-gray}{gray}{0.95}

\begin{document}

\bibliographystyle{plain}
\title[Torsion of rational elliptic curves over sextic fields]{On the torsion of rational elliptic curves\\ over sextic fields}

\author{Harris B. Daniels}
\address{Department of mathematics and statistics, Amherst College, MA 01002, USA}
\email{hdaniels@amherst.edu} 

\author{Enrique Gonz\'alez-Jim\'enez}
\address{Universidad Aut{\'o}noma de Madrid, Departamento de Matem{\'a}ticas, Madrid, Spain}
\email{enrique.gonzalez.jimenez@uam.es}

\subjclass[2010]{Primary: 11G05; Secondary: 14H52,14G05}
\keywords{Elliptic curves, torsion subgroup, rationals, sextic fields.}
\thanks{The first author was partially  supported by the grant MTM2015--68524--P.}

\date{\today}

\begin{abstract}

Given an elliptic curve $E/\Q$ with torsion subgroup $G = E(\Q)_{\tors}$ we study what groups (up to isomorphism) can occur as the torsion subgroup of $E$ base-extended to $K$, a degree 6 extension of $\Q$. We also determine which groups $H = E(K)_{\tors}$ can occur infinitely often and which ones occur for only finitely many curves. This article is a first step towards a complete classification of torsion growth over sextic fields.
\end{abstract}

\maketitle

\section{Introduction}

A fundamental theorem in the field of arithmetic geometry states that given a number field $K$ and an elliptic curve $E$ defined over $K$, the set of $K$-rationals points on $E$, denoted $E(K)$, can be given the structure of a finitely generated abelian group. Further, it is well-known that the torsion subgroup of this group $E(K)_{\tors}$ is isomorphic to $\Z/m\Z \times \Z/mn\Z$ for some positive integers $m$ and $n$. For the sake of brevity we will write $\Z/n\Z =(n)$ and $\Z/n\Z\times\Z/mn\Z=(m,mn)$ and call $(m,mn)$ the torsion structure of $E$ over $K$. By abuse of notation, we write $E(K)_{\tors}=(m,mn)$. 

One of the main question in the theory of elliptic curves is the following:
\begin{question}
Given a positive integer $d$ what groups (up to isomorphism) arise as the torsion subgroup of an elliptic curve defined over a number field of degree $d$ over $\Q$?
\end{question}

Beginning with Mazur's classification of torsion structures over $\Q$ in \cite{Mazur1978}, mathematicians have spent conside\-rable time and effort answering different facets of this question. Mazur's result asserts that the possible torsion structures over $\Q$ belong to the set:
$$
\Phi(1) = \left\{ (n) \; : \; n=1,\dots,10,12 \right\} \cup \left\{ (2 ,2m) \; : \; m=1,\dots,4 \right\}.
$$
Over quadratic fields, the classification of torsion structures was completely settled in a series of papers by Kamienny \cite{K92} and Kenku and Momose \cite{KM88}. Currently there is no complete classification for the case over cubic extension in the literature, but there are many significant results towards such a classification. In order to describe what is known we define the following notation:
\begin{itemize}

\item Let $\Phi(d)$ be the set of groups up to isomorphism that occur as the torsion structure of an elliptic curve defined over a number field of degree $d$.
\item Let $\Phi^\infty(d)\subseteq \Phi(d)$ be the set of groups that arise for infinitely many $\Qbar$-isomorphism classes of elliptic curves defined over number fields of degree d,
\end{itemize}
For degree $d=1,2$, each group in $\Phi(d)$ occur for infinitely many elliptic curves and so $\Phi^\infty(d) = \Phi(d)$. While determining the set $\Phi(d)$ is still open for $d\geq 3$, the uniform boundedness of torsion on elliptic curves, proved by Merel \cite{Merel} states that for any $d$, there exists a bound $B(d)$ depending only on $d$ such that $|G|\leq B(d)$ for all $G\in \Phi(d)$. Thus, the set $\Phi(d)$ is finite for any $d$. While $\Phi(d)$ is not completely known, $\Phi^\infty(d)$ is known for $d = 3,4$ thanks to the work of Jeon et al. \cite{JKP04,JKP06} and $d=5,6$ by Derickx and Sutherland \cite{DS17}. In this article we make use of the case when $d=6$:
\begin{align*}
\Phi^\infty(6) =\ &\{(n) \,|\, n=1,\dots, 22,24,26,27,28,30\}\ \cup\ \{(2,2m) \,|\,m=1,\dots, 10\}\\
&\cup\ \{(3,3m) \,|\,m=1,\dots, 4\}\ \cup \{(4,4),(4,8),(6,6)\}.
\end{align*}
Unlike in the cases when $d=1$ or $2$ when $d = 3,5,$ or $6$ we know that $\Phi^\infty(d) \subsetneq \Phi(d)$ and in the case when $d=4$, it is not known if $\Phi(4)$ and $\Phi^\infty(4)$ coincide. To illustrate the case when $d=3$, Najman showed in \cite{N15a} that the elliptic curve  with Cremona label \href{http://www.lmfdb.org/EllipticCurve/Q/162b1}{\texttt{162b1}} has a point of order 21 defined over the field $\Q(\zeta_{9})^+ = \Q(\zeta_9+\zeta_9^{-1})$ where $\zeta_9$ is a primitive $9$-th root of unity and therefore $(21)\in \Phi(3)$, but $(21)\not\in \Phi^\infty(3)$. For degrees $d=5,6$, van Hoeij \cite{hoeij} gives examples of elliptic curves defined over a number field of degree $d$ (not coming from elliptic curves over $\Q$) with torsion structure not belonging to $\Phi^\infty(d)$. With these examples in mind we give the following definition:
\begin{itemize}
\item Let $J(d)\subseteq \overline{\Q}$ be the finite set defined by the following property: $j\in J(d)$ if and only if there exists a number field $K$ of degree $d$, and an elliptic curve $E/K$ with $j(E)=j$, such that $E(K)_{\tors}$ is isomorphic to a group in $\Phi(d)$ that is not in $\Phi^\infty(d)$. We let $J_\Q(d) = J(d) \cap \Q$ be the subset of $J(d)$ where we restrict to the case of elliptic curves $E$ defined over $\Q$. 
\end{itemize}
Since $\Phi(d) = \Phi^\infty(d)$ when $d=1$ or $2$ we have that $J(1) = J(2) = \emptyset$ and Najman's examples shows that $-140625/8\in J_\Q(3)$.

The question at the center of this paper is how torsion subgroups of elliptic curves can grow when they are considered over larger fields of definition. In particular, we will consider how the torsion subgroup of elliptic curves defined over $\Q$ can change when we consider them over a sextic extension of $\Q$. The techniques in the paper are closely related to those found in \cite{GJLR18} and following that paper we give the following definitions. 
\begin{itemize}
\item Let $\Phi_\Q(d)$  be the subset of $\Phi(d)$  such that $H\in \Phi_\Q(d)$ if there is an elliptic curve $E/\Q$ and a number field $K$ of degree $d$ such that $E(K)_{\tors}= H$. Similarly, the set $\Phi_\Q^\infty(d)$ is contained inside the set $\Phi^\infty(d)$.
\item Let $\Phi_{\Q}^{\star}(d)$ be the intersection of the sets $\Phi_\Q(d)$ and  $\Phi^{\infty}(d)$. 
\item Fixed $G \in \Phi(1)$, let $\Phi_\Q(d,G)$ be the subset of $\Phi_\Q(d)$ such that $E$ runs through all elliptic curves over $\Q$ such that $E(\Q)_{\tors}= G$. Also, let $\Phi^{\star}_\Q(d,G)=\Phi_\Q(d,G)\cap \Phi^{\infty}(d)$.
\item Let $R_\Q(d)$ be the set of all primes $p$ such that there exists a number field $K$ of degree $d$ and an elliptic curve $E/\Q$ such that $E$ has a point of order $p$ defined over $K$. 
\end{itemize}
In the case when  $d=2,$ or $3$, the sets $\Phi_{\mathbb Q}(d)$ have been completely described by Najman \cite{N15a}:
$$
\begin{array}{l}
\Phi_{\mathbb Q}(2)=\left\{ (n) \; | \; n=1,\dots, 10, 12, 15, 16\right\}\cup \left\{ (2,2m) \; | \; m=1,\dots ,6 \right\}\cup \left\{ (3,3),\,(3,6),\,(4,4) \right\}, \hbox{ and } \\
\Phi_{\mathbb Q}(3)=\left\{ (n) \; | \; n=1,\dots, 10, 12, 13, 14, 18, 21\right\}\cup \left\{ (2,2m) \; | \; m=1,2,3,4,7 \right\}.
\end{array}
$$
The sets $\Phi_\Q(4)$,  $\Phi_\Q(5)$, and $\Phi_\Q(7)$ have been completely classifies in \cite{Chou16, GJN16}, \cite{GJ17}, and \cite{GJN16} respectively. Moreover in \cite{GJN16} it has been\footnote{Let $E/\Q$ be an elliptic curve and $K/\Q$ a number field of degree $d$ whose prime divisors are greater than $7$, then $E(K)_{\tors}=E(\Q)_{\tors}$ (see Remark 7.5 \cite{GJN16}).}  established  $\Phi_\Q(d)=\Phi(1)$ for any positive integer $d$ whose prime divisors are greater than $7$. \\

For any $G\in\Phi(1)$ the set $\Phi_\Q(d,G)$ has been determined for $d=2$ in \cite{GJT14}, for $d=3$ in \cite{GJNT16}, for $d=4$ in \cite{GJLR18}, for $d=5$ in \cite{GJ17}, for $d=7$ in \cite{GJN16} and for any $d$ whose prime divisors are greater than $7$ in \cite{GJN16}.\\

Further, in \cite{GJN16} the second author and Najman determine all the possible degrees of $[\Q(P):\Q]$, where $P$ is a point of prime order $p$ for a set of density $1535/1536$ of all primes and in particular for all $p<3167$. In particular, we have that
$$
R_\Q(6)=\{2,3,5,7,13\}.
$$
Finally, we points out that obtaining the set $\Phi_{\Q}^{\star}(4)$ was a main tool to the determination of the whole classification of $\Phi_{\Q}(4)$. Thus, a first attempt to obtain $\Phi_\Q(6)$ is to obtain $\Phi^*_\Q(6)$ and that is our first main result.

\begin{theorem}\label{main1} The set $\Phi^\star_\Q(6)$ is given by
\begin{align*}
\Phi_{\mathbb Q}^{\star}(6)=\ & \left\{ (n) \; | \; n=1,\dots,21,\,n\ne 11,17,19,20\right\}\cup\left\{(30) \right\}  \\ 
& \cup \left\{ (2,2n) \; | \; n=1,\dots,7,9 \right\} \cup  \left\{ (3,3n) \; | \; n=1,\dots,4 \right\} \cup  \left\{ (4,4),(6,6) \right\},
\end{align*}
and $\Phi_\Q^\infty(6)=\Phi_\Q^\star(6)\setminus \{(15),(21),(30) \}$. In particular, if $E/\Q$ is an elliptic curve with $j(E)\notin J_\Q(6)$, then $E(K)_{\tors}\in \Phi_\Q^\star(6)$, for any number field $K/\Q$ of degree $6$. Moreover, if $E/\Q$ is an elliptic curve with $E(K)_{\tors}= H\in \{(15),(21),(30) \}$  over some sextic field $K$, then 
\begin{romanenum}
\item $H=(21)$: $j(E) \in \{3^3\cdot 5^3/2,-3^2\cdot5^3\cdot 101^3/2^{21},-3^3\cdot 5^3\cdot 382^3/2^7,-3^2\cdot 5^6/2^3\}$.
\item $H=(15)$: then $E$ has Cremona label \href{http://www.lmfdb.org/EllipticCurve/Q/50a3}{\texttt{50a3}},\href{http://www.lmfdb.org/EllipticCurve/Q/50a4}{\texttt{50a4}},\href{http://www.lmfdb.org/EllipticCurve/Q/50b1}{\texttt{50b1}},\href{http://www.lmfdb.org/EllipticCurve/Q/50b2}{\texttt{50b2}},\href{http://www.lmfdb.org/EllipticCurve/Q/450b4}{\texttt{450b4}}, or \href{http://www.lmfdb.org/EllipticCurve/Q/450b3}{\texttt{450b3}}.
\item $H=(30)$: then $E$ has Cremona label \href{http://www.lmfdb.org/EllipticCurve/Q/50a3}{\texttt{50a3}},\href{http://www.lmfdb.org/EllipticCurve/Q/50b1}{\texttt{50b1}},\href{http://www.lmfdb.org/EllipticCurve/Q/50b2}{\texttt{50b2}}, or \href{http://www.lmfdb.org/EllipticCurve/Q/450b4}{\texttt{450b4}}.
\end{romanenum}
\end{theorem}

Our second result determines $\Phi_\Q^\star(6,G)$ for each $G \in\Phi(1)$:

\begin{theorem}\label{main2}
For each $G \in \Phi(1)$, the set $\Phi^{\star}_\Q(6,G)$ is given in the following table:
$$
\begin{array}{|c|c|}
\hline
G & \Phi^{\star}_\mathbb{Q} \left(6,G \right)\\
\hline
(1) & \begin{array}{c} \left\{
 (1)\,,\,
 (2)\,,\,
 (3)\,,\,
 (4)\,,\,   
 (5)\,,\,
 (6) \,,\,
 (7) \,,\,
 (9) \,,\,
 (10)\,,\, 
 (12)\,,\,
 (13)\,,\,
 (14)\,,\,
 (15)\,,\, \right.\\ \left.
 (18)\,,\,
 (21)\,,\, 
 (2,2)\,,\, 
 (2,6)\,,\, 
 (2,10)\,,\, 
 (2,14)\,,\, 
 (2,18)\,,\, 
 (3,3)\,,\, 
 (3,9)\,,\, 
 (4,4)\,,\, 
 (6,6)\, \right\} \\   \end{array}\\ 

\hline
(2) &\begin{array}{c} \left\{ (2)\,,
(4)\,,\,
(6)\,,\,
(8)\,,\,
(10)\,,\,
(12)\,,\,
(14)\,,\,
(16)\,,\,
(18)\,,\,
 (2,2)\,,\, \right.\\ \left.
  (2,6)\,,\,
  (2,10)\,,\,
  (2,14)\,,\, 
  (2,18)\,,\,
  (3,6)\,,\,
  (3,12)\,,\,
   (6,6)\, \right\} 
   \end{array}\\ 
\hline
(3) & \begin{array}{c}
\left\{ 
(3)\,,\, 
(6)\,,\,
(9)\,,\, 
 (12)\,,\, 
(15)\,,\,
(21)\,,\, 
(30)\,,\, 
(2,6)\,,\,  \right.\\ \left.
(3,3)\,,\,
(3,6)\,,\,
(3,9)\,,\,
{  (6,6)}\,\right\} \\
   \end{array}\\ 

\hline
(4) & 
\left\{ 
(4) \,,\,  
(8)\,,\,
 {  (12)}\,,\,
{   (2,4)}\,,\,
  (2,8)\,,\, 
  (2,12)\,,\,
  (3,12)\,,\,
  (4,4)\, \right\} \\
\hline
(5) & \left\{ (5)\,,\, (10)\,,\, {  (15) }\,,\,(30)\,,\,(2,10)\, \right\} \\
\hline
(6) & \left\{ (6)\,,\, 
(12)\,,\,
(18)\,,\,
  {    (2,6)}\,,\,
  (2,18)\,,\,
   {   (3,6) }\,,\,
   {   (3,12) }\,,\,
  (6,6) \, \right\} \\
\hline
(7) & \left\{ (7)\,,\, 
(14)\,,\,    {   (2,14) }
  \,\right\} \\
\hline
(8) & \left\{ (8)\,,\,(16)\,,\,
   (2,8)\, \right\} \\
\hline
(9) & \left\{ (9)\,,\,(18)\,,\,(2,18)\,,\,
 (3,9)
 \, \right\} \\
\hline
(10) & \left\{ (10)\,,\, 
 {   (2,10)}\, \right\} \\
\hline
(12) & \left\{ (12)\,,\, 
 {   (2,12)}\,,\, 
 {   (3,12)} \right\} \\
\hline
{(2,2)}& 
\left\{ (2,2)\,,\,
 (2,4)\,,\,
 {   (2,6)} \,,\,
  (2,8)\,,\,  
   (2,12)\,,\,
   (6,6)\, \right\}
\\
\hline
(2,4) & \left\{ (2,4)\,,\,
 (2,8)\,,\,
   (4,4)\, \right\} \\
\hline
(2,6) & \left\{ (2,6)\,,\, (2,12)\,,\,(6,6)\, \right\} \\
\hline
(2,8) & \left\{ (2,8)\, \right\} \\
\hline
\end{array}
$$
\end{theorem}

\begin{remark}\label{spo}
In \cite{GJN18}, the second author and Najman show that $(4,12)\in \Phi_{\mathbb Q}(6)\setminus \Phi_{\mathbb Q}^{\star}(6)$. In particular, $(4,12)\in \Phi(6)\setminus \Phi^{\infty}(6)$.
Moreover, they show that if $E/\Q$ is an elliptic curve and $K/\Q$ is a sextic extension such that $E(K)_{\tors}=(4,12)$ then the Cremona label of $E$ is \href{http://www.lmfdb.org/EllipticCurve/Q/162d1}{\texttt{162d1}} or \href{http://www.lmfdb.org/EllipticCurve/Q/1296h1}{\texttt{1296h1}} and $K=\Q(\alpha,i)$ where $\alpha^3-3\alpha-4$. In particular, these elliptic curve are $\overline{\Q}$-isomorphic and $j(E)=109503/64\in J_\Q(6)$.
\end{remark}

\begin{corollary}
If $J_\Q(6)=\{109503/64\}$, then $\Phi_{\Q}(6)=\Phi_{\mathbb Q}^{\star}(6)\cup \{(4,12)\}$ and 
$\Phi_{\Q}(6,G)=\Phi_{\mathbb Q}^{\star}(6,G)$ if $G\ne (1),(3)$ or $\Phi_{\Q}(6,G)=\Phi_{\mathbb Q}^{\star}(6,G)\cup \{(4,12)\}$ if $G=(1)$ or $ (3)$.
\end{corollary}

\begin{conjecture}
$\Phi_{\mathbb Q}(6)=\Phi_{\mathbb Q}^{\star}(6)\cup \{(4,12)\}=\Phi_{\mathbb Q}(2)\cup \Phi_{\mathbb Q}(3)\cup\{(30),(2,18),(3,9),(3,12),(4,12),(6,6)\}$.
\end{conjecture}

\subsubsection*{Notation} 
Any specific elliptic curves mentioned in this paper will be referred to by Cremona label \cite{cremonaweb} and a link to the corresponding LMFDB page \cite{lmfdb} will be included for the ease of the reader. Conjugacy classes of subgroups of $\GL_2(\Z/p\Z)$ be referred to by the labels introduced by Sutherland in \cite[\S 6.4]{Sutherland2}. We write $G=H$ (or $G\subseteq H$) for the fact that $G$ is isomorphic to $H$ (or to a subgroup of $H$ resp.) without further detail on the precise isomorphism. By abuse of notation, we write $\Z/n_1\Z\times\cdots\times\Z/n_r\Z=(n_1,\dots,n_r)$.

\begin{remark} 
All of the computations in this paper have been performed using Magma \cite{magma} and some of the code has been taken from \cite{Q3}. All of the code needed to reproduce these computations can be found at \cite{magmascripts}. 
\end{remark}

\subsection{Acknowledgments} The authors would like to thank \'Alvaro Lozano-Robledo and Jeremy Rouse for helpful conversations while working on this project. We would also like to the the referee for helpful comments on an earlier draft of this paper and the editor for an efficient editorial process. 

\section{Background Information}\label{sec:aux}

Central to proving the main results of this paper is understanding fields of definition of the points of order $n$ on an elliptic curve defined over $\Q$. In order to do this we introduce an auxiliary object called the mod $n$ Galois representations associated to the torsion points of elliptic curves. First, fix an algebraic closure $\Qbar$ of $\Q$. Given an elliptic curve $E/\Q$ and an integer $n\geq 2$ we let $E[n] = \{ P \in E(\Qbar) \mid [n]P = \mathcal{O}\}$ be the subgroup of $E(\Qbar)$ consisting of the point of order dividing $n$. The absolute Galois group $\Gal(\Qbar/\Q)$ act coordinate-wise  on the point of $E[n]$ and this action induces a representation 
\[
\rho_{E,n}\colon \Gal(\Qbar/\Q) \to \Aut(E[n]).
\]
It is a classical result that $E[n]$ is a free rank two $\Z/n\Z$-module and thus fixing a basis $\{P,Q\}$ for $E[n]$ we have that $\Aut(E[n])$ is isomorphic to a subgroup of $\GL_2(\Z/n\Z)$. Therefore, we can write the Galois representation as 
\[
\rho_{E,n}\colon \Gal(\Qbar/\Q) \to \GL_2(\Z/n\Z).
\]
where the image of $\rho_{E,n}$, denoted $G_E(n)$, is determined up to conjugacy (i.e.~a choice of basis for $E[n]$). If we let $\Q(E[n]) = \Q(\{x,y\mid (x,y) \in E[n]  \})$ be the field of definition of the $n$-torsion points on $E$, then $\Q(E[n])/\Q$ is a Galois extension and since $\ker\rho_{E,n} = \Gal(\Qbar/\Q(E[n]))$, from elementary Galois theory we have that $\Gal(\Q(E[n])/\Q)\simeq G_E(n)$.

Given a point $R\in E[n]$, we will denote the $x$- and $y$-coordinates of $R$ by $x(R)$ and $y(R)$ respectively and the field of definition for $R$ will be be denoted by $\Q(R) = \Q(x(R),y(R))$. From Galois theory, we know that there is a subgroup $\mathcal{H}_R$ of $\Gal(\Q(E[n])/\Q)$ such that $\Q(R)$ is the subfield of $\Q(E[n])$ fixed by $\mathcal{H}_R$. Letting $H_R = \rho_{E,n}(\mathcal{H}_R)$  we have the following two facts.
\begin{enumerate}
\item $[\Q(R):\Q] = [G_E(n) : H_R]$;
\item If $\widehat{\Q(R)}$ is the Galois closure of $\Q(R)$ in $\Qbar$ and $N_{G_E(n)}(H_R)$ is the normalizer of $H_R$ in $G_E(n)$, then $\Gal(\widehat{\Q(R)}/\Q)\simeq G_E(n)/N_{G_E(n)}{(H_r)}$.
\end{enumerate}
Practically, if we are given $G_E(n)$ up to conjugation, one can deduce many algebraic properties of $\Q(E[n])$. In particular, since $E[n]$ is a free rank two $\Z/n\Z$-module the $n$-torsion points of $E$ can be (non-canonically) identified with elements of $(\Z/n\Z)^2$ by taking their coordinate vectors with respect to a fixed basis. With this identification, the group $H_R$ is exactly the stabilizer of the coordinate vector of $R$ with respect to the action of $G_E(n)$ on $(\Z/n\Z)^2$. With this we can compute all possible degrees of $\Q(R)/\Q$ where $R$ is a point of exact order $n$, by computing the index of the stabilizer of each element of $(\Z/n\Z)^2$ of order $n$ inside $G_E(n)$. Lastly, as consequence of the Weil pairing we always have that the determinant map $\det\colon G_E(n) \to (\Z/n\Z)^\times$ is surjective and because $\Gal(\Qbar/\Q)$ contains complex conjugation, $G_E(n)$ must contain an element of trace $0$ and determinant $-1$. For more details about this see \cite[Proposition 2.2]{zywina1}

\subsection{Classifications of the Possible Images of Galois Representations} One of the first major results that will be regularly used is the complete classification of elliptic curves with cyclic isogenies. We remind the reader that an elliptic curve $E/K$ has an $n$-{isogeny} if there is a degree $n$ map $\phi\colon E\to E'$ such that $\ker\phi$ is a cyclic subgroup of $E[n]$ of order $n$. If $E/K$ has a cyclic $n$-isogeny, we know that $E[n]$ contains a Galois-stable cyclic subgroup of order $n$, and thus $G_E(n)$ is conjugate to a subgroup of the Borel group of upper triangular matrices in $\GL_2(\Z/n\Z)$. For the sake of concision, whenever $E/\Q$ has a cyclic $n$-isogeny defined over $\Q$, we will simply say that $E$ has a {\em  rational $n$-isogeny}.

\begin{theorem}[Mazur \cite{Mazur1978} and Kenku \cite{kenku2, kenku3, kenku4, kenku5}]\label{isogQ}
Let $E/ \Q$ be an elliptic curve with a rational $n$-isogeny. Then
\begin{equation}\label{isogenies}
n\in \{1,\ldots 19,21,25,27,37,43,67,163\}.
\end{equation}
Further, there are infinitely many $\overline{\Q}$-isomorphism classes of elliptic curves with a rational $n$-isogeny for all $n\in\{1,\ldots,10, 12,13,16,18,25\}$ and only finitely many for all the other $n$ listed in \eqref{isogenies} and if $n\in \{14,19,27,43,67,163\}$ then $E$ has complex multiplication.
\end{theorem}

The next results that we will need are related to the classification of possible images of Galois representations associated to rational elliptic curves of various levels. The first set of results are contained in \cite{zywina1} where Zywina classifies (among other things) the complete list of possible images of the mod $\ell$ Galois representations associated to rational elliptic curves for all $\ell\leq 13.$ We only need the classification up to $\ell=13$ because as mentioned above $R_\Q(6) = \{ 2,3,5,7,13 \}$. For all of the possible images except three, Zywina gives a complete description of the elliptic curves over $\Q$ whose mod $p$ Galois representation has image conjugate to a subgroup of a given group. Two of the three remaining cases were handled in \cite{curse} by Balakrishnan et al.~ using the Chabauty-Kim method to determine all the rational points on the ``cursed'' genus 3 modular curve.  While a classification for the last remaining image (\texttt{13S4} in Sutherland's notation) is still just a conjecture, since any curve with this image does not have a point of order 13 defined over a sextic extension for our purposes the classification is complete. The second set of results we will use is contained in \cite{RZB} where Rouse and Zureick-Brown give a complete account of the possible 2-adic\footnote{The $p$-adic Galois representations associated to an elliptic curve are constructed by taking the inverse image of the mod $p^n$ Galois representations.} images of Galois representations of elliptic curves without complex multiplication defined over $\Q$. Besides listing each of the possible images they give a complete description of  the associated moduli spaces as well, this includes describing all of the rational points on each modular curves. \\

Finally a specific result about rational isogenies for torsion growth over sextic field.

\begin{lemma}\label{lemIsog}
Let $E/\Q$ be an elliptic curve without complex multiplication, $K/\Q$ a sextic field and $P_p\in E(K)_{\tors}$ a point of odd prime order $p$. Then $E$ has a rational $p$-isogeny, except if $E$ has Cremona label \href{http://www.lmfdb.org/EllipticCurve/Q/2450ba1}{\texttt{2450ba1}} or \href{http://www.lmfdb.org/EllipticCurve/Q/2450bd1}{\texttt{2450bd1}}, and $p=7$, where there is not rational $7$-isogenies. Moreover, in those last cases, the unique sextic fields where the torsion grows are $K=\Q(E[2])$ and $K'=\Q(P_7)$ ($K'/\Q$ is a non-Galois), where $E(K)_{\tors}=(2,2)$ and $E(K')_{\tors}=(7)$ respectively.
\end{lemma}

\begin{proof}
We have that $p\in R_\Q(6)=\{2,3,5,7,13\}$.  Looking at the Table \ref{TABLE} we see that, unless $G_E(7)$ is conjugate to \texttt{7Ns.2.1}, is must be that $E$ has a rational $p$-isogeny.  For the case \texttt{7Ns.2.1}, \cite[Theorem 1.5 (iii)]{zywina1}  says that $E$ has Cremona label \href{http://www.lmfdb.org/EllipticCurve/Q/2450ba1}{\texttt{2450ba1}} or \href{http://www.lmfdb.org/EllipticCurve/Q/2450bd1}{\texttt{2450bd1}}. For those two curves we have that they have no rational isogenies  and the unique prime where $\rho_{E,p}$ is non-surjective is $p=7$, then by Table \ref{TABLE} we have that the sextic fields where the torsion grows are $K=\Q(E[2])$ and $K'=\Q(P_7)$ and computing the torsion over that number fields we obtain the desired result.
\end{proof}

\subsection{Elliptic Curves with Complex Multiplication} Let $\Phi^{\cm}(d)$ be the set consisting of the torsion subgroups of elliptic curves with complex multiplication {(or CM for short)} defined over field of degree $d$. Table \ref{tableCM} lists the sets $\Phi^{\cm}(d)$ such that $d\,|\,6$ since these will be that ones that we use in this article. Proofs of the results in Table \ref{tableCM} can be found in \cite{Clark2014,FSWZ90,MSZ89,Olson74,PWZ97}.

\begin{table}[h]
\begin{tabular}{|c|c|}
\hline
$d$& $\Phi^{\cm}(d)$ \\
\hline
  $1$ & $\left\{ (1)\,,\,  (2)\,,\,  (3)\,,\,  (4)\,,\,  (6)\,,\, (2,2)\right\}$\\
   \hline
 $2$ & $\Phi^{\cm}(1)\cup \left\{ (7)\,,\,(10)\,,\,(2,4)\,,\,(2,6)\,,\,(3,3)\,\right\}$\\
   \hline
 $3$ & $\Phi^{\cm}(1)\cup  \left\{\,(9)\,,\,(14)\,\right\}$\\
 \hline
$6$ & $\Phi^{\cm}(2)\cup\Phi^{\cm}(3)\cup
 \left\{ \,(18)\,,\,(19)\,,\,(26)\,,\,(2,14)\,,\,(3,6)\,,\,(3,9)\,,\, (6,6)\,\right\}$\\
 \hline   
\end{tabular}
\caption{$\Phi^{\cm}(d)$, for $d\,|\,6$.}\label{tableCM}
\end{table}

\section{Proof of the Main Auxiliary Results}\label{sec_aux} 
The determination of $\Phi^{\star}_\Q(6)$ and $\Phi^{\star}_\Q(6,G)$ will rest on Proposition \ref{teo} and \ref{teoCM} for the case non-CM and CM, respectively.

\begin{proposition}\label{teo}
Let $E/\Q$ be an elliptic curve without complex multiplication and $K/\Q$ a sextic number field such that $E(\Q)_{\tors}= G$ and $E(K)_{\tors}= H$.
\begin{alphenum}

 \item\label{t1}$11,17$ and $19$ do not divide the order of $H$.

\item\label{t2} Let $G_2$ (resp.~$H_2$) denote the $2$-primary part of $G$ (resp.~$H$) then the only possible $2$-primary torsion growth are given in Table \ref{Tab:2-prim}. For each entry in the table a {\rm ($-$)} indicates that the growth from $G_2$ to $H_2$ cannot happen. If the growth from $G_2$ to $H_2$ is possible, we give the modular curve in the notation of \cite{RZB} that parameterizes elliptic curves with this growth. That is, $E(\Q)_{\tors}$ contains a subgroup isomorphic to $G_2$ and there is a sextic extension $K/\Q$ such that $E(K)_{\tors}$ contains a subgroup isomorphic to $H_2$ if and only if $E$ corresponds to a rational point on the given modular curves. 
\begin{table}
\renewcommand{\arraystretch}{1.2}
\begin{tabular}{|c|c|c|c|c|c|c|c|c|c|}
\hline
\backslashbox{$G_2$}{$H_2$}
 & $(1)$ & $(2)$ & $(4)$ & $(8)$ & $(16)$ & $(2,2)$ & $(2,4)$& $(2,8)$& $(4,4)$\\
\hline
$(1)$ & $X_1$ & $X_1$  & $X_{20}$  & $-$ & $-$ & $X_{1}$  & $-$ & $-$ & $X_{20b}$  \\
\hline
$(2)$ & $-$ & $X_6$ & $X_{13}$ & $X_{102},X_{36a}$ & $X_{235m}$ & $X_6$ & $-$ & $-$ & $-$ \\
\hline
$(4)$ & $-$ & $-$ & $X_{13h}$ & $X_{36n}$ & $-$ & $-$ & $X_{13h}$ & $X_{102k}$ & $X_{60d}$ \\
\hline
$(8)$ & $-$ & $-$ & $-$ & $X_{102p}$ & $X_{235l}$ & $-$ & $-$ & $X_{102p}$ & $-$ \\
\hline
$(2,2)$ & $-$ & $-$ & $-$ & $-$& $-$ & $X_{8}$ & $X_{25},X_{8d}$ & $X_{193},X_{96q},X_{98o}$  & $-$ \\
\hline
$(2,4)$ & $-$ & $-$ & $-$ & $-$ & $-$ & $-$ & $X_{25n}$ & $X_{96t},X_{98e}$ & $X_{58i}$ \\
\hline
$(2,8)$ & $-$ & $-$ & $-$ & $-$ & $-$ & $-$ & $-$ & $X_{193n}$ & $-$ \\
\hline
\end{tabular}\caption{ Classification of the possible growth in 2-primary component over a sextic field. }\label{Tab:2-prim}
\end{table}

\item\label{t5quartic} If $(4)\subseteq G$, then $(20)\not\subseteq H$.
\item\label{t6quartic} If $(8)\subseteq G$, then $(24)\not\subseteq H$.
\item \label{t7quartic} If $(2,2)\subseteq G$, then $(2,10)\not\subseteq H$.
\item\label{t8quartic} If $(2,4)\subseteq G$, then $(2,12)\not\subseteq H$.
\item\label{t24} If $G =(12)$, then $H\ne (24)$.
\item\label{t2x14} If $G=(2,2)$, then $(2,14)\not\subseteq H$.
\item\label{t6x6} If $G = (1)$ and $(3,6)\subseteq H$, then $(6,6)\subseteq H$.
\item\label{t3x12} If $G = (3)$, then $H\ne (3,12)$.
\item\label{t20} $H\ne (20)$.
\item\label{t18} If $G = (3),$ then $H\ne (18)$.
\item\label{t2x18} If $(2,2)\subseteq G$, then $(2,18)\not\subseteq H$.
\item\label{3t2x18} If $G= (3)$, then $H\ne (2,18)$.
\item\label{t21} If $G=(7)$, then $(21)\not\subseteq H$.
\item\label{2t24} If $G =(2), (4)$ or $(6)$, then $H \ne (24)$.
\item\label{t26} $(26)\not\subseteq H$.
\item\label{t27} $H \ne (27)$.
\item\label{t28} $H\ne (28)$.
\item\label{t30} If $G\ne (3)$ or $(5)$, then $H\ne (30)$.

\end{alphenum}
\end{proposition}

\begin{proof}

(\ref{t1}) By \cite{GJN16} we have $11,17,19\notin R_{\Q}(6)=\{2,3,5,7,13\}$.\\

(\ref{t2}) As mentioned in Section \ref{sec:aux}, Rouse and Zureick-Brown completely classify all of the possible 2-adic images of Galois representations associated to elliptic curves without CM defined over $\Q$ in \cite{RZB}. For each of the possible images the second author together with Lozano-Robledo computed the degree of the field of definition of the $(2^i,2^{i+j})$ torsion for $i+j \leq 6$ in \cite{GJLR17} and recoded the data in a text file titled \texttt{2primary\_Ss.txt} (cf. \cite{GJLR17}). Using the results of \cite{GJLR17} we write a program that takes as its input a degree $d$ and returns an associative array whose keys are all the possible 2-primary parts of $E(\Q)_{\tors}$ and $E(K)_{\tors}$ of elliptic curves defined over $\Q$ and base-extended to a degree $d$ number field $K$ over $\Q$. The objects associated with each of these curves are the labels of the modular curves (in the notation of \cite{RZB}) that parameterize each of the possible growths in two primary components. This algorithm and its output can be found in \cite{magmascripts} in the file labeled \texttt{RZB\_Search.txt}.\\

(\ref{t5quartic}), (\ref{t6quartic}), (\ref{t7quartic}) and (\ref{t8quartic}): See Remark below Theorem 7 of \cite{GJLR18}.\\

(\ref{t24}) The $2$-divisibility method \cite[Section 3]{JKL13} asserts that if a point $Q$ satisfies $2Q=P$ then $[K(Q):K]\le 4$ (see \cite[Remark 3.2]{JKL13}). In the particular case of $K=\Q$ and $P$ of order $12$, we have that $(24)$ is a subgroup of some group in $\Phi_\Q(d)$ for $d\le 4$. But we know that this can only happen in a quartic extension \cite[Corollary 8.7]{GJN16} of $\Q$ and thus cannot happen over a sextic extension. \\

(\ref{t2x14}) Suppose that $G=(2,2)$ and $(2,14) \subseteq H$. Since $(2,14)$ is not a subgroup of any group in $\Phi_{\Q}(d,(2,2))$ for $d=2,3$ by \cite[Theorem 2]{GJT14} and \cite[Theorem 1.2]{GJNT16} respectively, we have that $E$ gain a point of order $7$ over a sextic field, and not over any number field of degree less than $6$. By Lemma \ref{lemIsog}
we have that, unless $G_E(7)$ is conjugate to \texttt{7Ns.2.1}, $E$ has a rational $7$-isogeny.  In the former case we have $G=(1)$ therefore we can exclude it for the rest of the proof. Now, since $G=(2,2)$, then $E$ is $2$-isogenous (over $\Q$) to two curves $E'$ and $E''$, such that $E$, $E'$, and $E''$ are all non-isomorphic pairwise. Further, there is a rational $4$-isogeny from $E'$ to $E''$ that is necessarily cyclic. Moreover, since $E$ has a rational $7$-isogeny, if follows that $E'$ also has a rational $7$-isogeny, and therefore $E'$ would have a rational $28$-isogeny which is impossible by Theorem \ref{isogQ}.\\

(\ref{t6x6})  Suppose that $G = (1)$ and $(3,6) \subseteq H$. Checking Table \ref{TABLE} we see that the only way that 
we can go from trivial torsion to having full 3-torsion over a sextic 
extension of $\Q$ is for $G_E(3)$ conjugate to $\texttt{3B.1.2}$. Therefore, we can 
pick a basis $P_3,P_3'$ for $E[3]$ such that $[\Q(P_3):\Q]=2$ and 
$[\Q(P_3'):\Q]=3$ and $K = \Q(P_3,P_3')$ is an $S_3$-extension of $\Q$, since \texttt{3B.1.2} is isomorphic to $S_3$. 
Next, in order to pick up 
a point $P_2$ not defined over $\Q$, it must be that 
$[\Q(P_2):\Q] = 3$ where $G_E(2)$ is conjugate to $\texttt{2Cn}$ or $\rho_{E,2}$ is surjective. 
In the first case we have $\Q(P_2)=\Q(E[2])\subseteq K$. In the second 
case $\Q(E[2])$ is the Galois closure of $\Q(P_2)$, that is the $S_3$-extension 
of $\Q$ that contains $\Q(P_2)$. Therefore, $\Q(E[2]) = \Q(E[3])$. In both 
cases we obtain that $(6,6)\subseteq E(K)_{\tors}$. Therefore $H$ cannot equal $(3,6)$.\\

(\ref{t3x12}) If $G =(3)$ and $H = (3,12)$, then we know that $E$ must have a rational $3$-isogeny and from part (\ref{t2}) in order for $E/\Q$ to gain a point of order 4 over a sextic extension $E$ must correspond to a $\Q$-rational point on the modular curve $X_{20}$. Taking the fiber product of $X_0(3)$ and $X_{20}$ we get a singular genus $1$ curve $C$ whose desingularization is the elliptic curve $E'/\Q$ with Cremona label \href{http://www.lmfdb.org/EllipticCurve/Q/48a3}{\texttt{48a3}} and $E'(\Q)=(2,4)$. Inspecting the rational points on $C$ we get that there are $4$ non-singular non-cuspidal points corresponding to the $j$-invariants $109503/64$  and $-35937/4$. For each of these $j$-invariants there is exactly one $\Q$-isomorphism class that has a point of order $3$ defined over $\Q$. Representatives of these classes are the curves with Cremona label \href{http://www.lmfdb.org/EllipticCurve/Q/162a1}{\texttt{162a1}} and \href{http://www.lmfdb.org/EllipticCurve/Q/162d1}{\texttt{162d1}}. Checking the $3$-division fields of each of these curves, we see that neither gains full 3-torsion and a point of order 6 over the same sextic extension of $\Q$. \\

(\ref{t20}) If $(2)\subseteq G$, then $E$ has a rational point of order 2 and the 2-power division fields are all 2-extensions of $\Q$. Therefore, if $H = (20)$ it must be that if $P_4$ is the point of order 4 over $K$, then $\Q(P_4)$ is a quadratic extension of $\Q$. Further, by Table \ref{TABLE}, the only way that $E$ can have a point $P_5$ of order 5 over $K$ is if $\Q(P_5)$ is defined over a quadratic extension of $\Q$. If $K$ is a sextic extension then it must be that $\Q(P_4) \subseteq \Q(P_5)$ and $E$ actually has a point of order 20 defined over a quadratic extension which is impossible. 

Lastly, if $(2)\not\subset G$ and $H = (20)$, then $E$ gains a point $P_4$ of order 4 over a degree 3 or 6 field and from part (\ref{t2}), the only way this can happen is if $E$ corresponds to a rational point on the curve $X_{20}$ from \cite{RZB}. Again by Lemma \ref{lemIsog}, in order to gain a point of order 5 over a sextic extension, $E$ must have a rational 5-isogeny. Computing the fiber product of $X_{20}$ and $X_0(5)$ we get a genus $3$ hyperelliptic  curve $C$ with $\Aut(C)=(2,2)$. The automorphism group of $C$ is generated by the hyperelliptic involution and another automorphism of order two, call it $\phi$. The curve, $E'$ obtained by quotienting out by $\phi$ is the elliptic curve with Cremona label \href{http://www.lmfdb.org/EllipticCurve/Q/80a4}{\texttt{80a4}} which has $E'(\Q)=(4)$. Computing the preimage of the $4$ points on $E'$ we see that $C(\Q)$ consists of exactly $3$ points, two of which are singular and one is a cusp at infinity. Thus, there are no elliptic curves over $\Q$ that gain a point of order $4$ over a sextic and have a rational $5$-isogeny. \\

(\ref{t18}) Suppose towards a contradiction that $G = (3)$ and $H = (18)$ and let $P_2$ and $P_9$ be points in $E(K)$ of order 2 and 9 respectively. Since $E(\Q)[2]$ is trivial it must be that $[\Q(P_2):\Q]= 3$ and from \cite[Proposition 4.6]{GJN16} we have that $[\Q(P_9):\Q]= 2, 3$ or $6$. From \cite{GJT14} we know that $[\Q(P_9):\Q] \neq 2$ and from \cite{GJNT16} we know that $[\Q(P_9):\Q] \neq 3$ since there are no elliptic curves whose torsion grows from $(3)$ to $(18)$ over a cubic. So it must be that $[\Q(P_2):\Q]= 3$, $[\Q(P_9):\Q] =6$ and $K=\Q(P_9)$. So we search in Magma for subgroups of $\GL_2(\Z/9\Z)$ that would correspond to an elliptic curve with a rational point of order 3 and a point of order 9 defined over a degree 6 extension of $\Q$. We get that this can happen in $3$ different ways according to the Galois closure of $\Q(P_9)$. That is, $\Gal(\widehat{K}/\Q)\cong \cC_6,S_3,$ or $\cC_3\times S_3$. In the first two cases, we have that $\widehat{K} = K = \Q(P_9)$ and since $\Q(P_2)\subseteq \Q(P_9)$ we know that the Galois closure of $\Q(P_2)$, which is $\Q(E[2])$, is contained in $K$ and this is a contradiction to the assumption that $H = (18)$, and cyclic. Therefore, it must be that $\Gal(\widehat{K}/\Q) = \cC_3\times S_3$ and we have the following field diagram:
$$
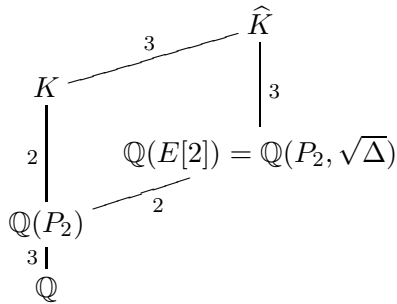
\xymatrix@=8pt{
  &   \widehat{K} \\
K\ar@{-}[ur]^3    &  \\
  &  \Q(E[2]) =\Q(P_2,\sqrt{\Delta})\ar@{-}[uu]_3 \\
\Q(P_2)\ar@{-}[uu]^2 \ar@{-}[ur]_2 &  \\
\Q\ar@{-}[u]^3    &  \\
}$$
where $\Delta$ is the discriminant of $E$ and $\Delta$ is not a square since $\Q(E[2]) \not\subseteq K$. But, since there is unique subgroup of $\cC_3\times S_3$ of index 2, we know that there is a unique quadratic extension of $\Q$ inside $\widehat{K}$ and by the above field diagram it must be that $\Q(\sqrt{\Delta})$ is inside of $K$. This means that $K = \Q(P_9) = \Q(P_2,\sqrt{\Delta}) =\Q(E[2])$ which is a Galois extension of $\Q$ giving us a contradiction. \\

(\ref{t2x18}) We have that $(2,18)$ is not a subgroup of any group in  $\Phi_{\Q}(d)$ for $d=2,3$. Therefore we have that there must exist a point $P_3$ of order $3$ such that $[\Q(P_3):\Q]=6$ and there is not other point of order $3$ defined over a number field of degree less than $6$. According to Table \ref{TABLE} this cannot happen.\\

(\ref{3t2x18}) Suppose towards a contradiction that $G = (3)$ and $H = (2,18)$. Let $P_9$ be the point of order 9 defined over $K$. In this case we have that $[\Q(P_9):\Q] = 2, 3$ or $6$ and $[\Q(E[2]):\Q] = 3$ or $6$. We can exclude the case $[\Q(P_9):\Q] = 2$ since $(9)$ is not the subgroup of some group in $\Phi_\Q(2,(3))$ (see \cite[Theorem 2]{GJT14}). From \cite{GJNT16} we know that both of these indices cannot be $3$. Suppose that $[\Q(P_9):\Q] =3$ and $[\Q(E[2]):\Q] = 6$. This means that $\Q(P_9)/\Q$ is a degree 3 subfield of $\Q(E[2])$ and so $E$ must have a point of order 2 defined over $\Q(P_9)$. In this case we would have $\Q(P_9)/\Q$ is a cubic extension where $E(\Q(P_9))_{\tors} = (18)$ but this can not happen from \cite[Theorem 1.2]{GJNT16}. Therefore it must be that $[\Q(P_9):\Q] = 6$ and $K = \Q(P_9)$. Just as in part (\ref{t18}) there are only 3 possible options for $\Gal(\widehat{K}/\Q)$, they are $\cC_6,S_3,$ and $\cC_3\times S_3$. The third case gives the exact same contradiction as in part (\ref{t18}), while if $\Gal(\widehat{K}/\Q) = \cC_6,$ or $S_3$, then $\Q(P_9)/\Q$ is a Galois extension and $\langle P_9 \rangle$ is a cyclic Galois stable subgroup of $E(K)$. Therefore $G_E(9)$ must be conjugate to a subgroup of the Borel subgroup of $\GL_2(\Z/9\Z)$, implying that $\langle P_9\rangle$ must be the kernel of a rational $n$-isogeny and $\Q(P_9)/\Q$ must be a cyclic extension. Therefore, $\Gal(K/\Q) =\cC_6$ and $E$ has a rational 9-isogeny. Since $\Q(E[2])\subseteq\Q(P_9)$ and $\Q(E[2])/\Q$ is Galois, it must be that $\Gal(\Q(E[2])/\Q)\simeq\cC_3$ which can only happen if $E$ has square discriminant. Checking in Magma we see that there are no curves with a rational 9-isogeny and square discriminant. \\

(\ref{t21}) By \cite[Theorem 2]{GJT14} and \cite[Theorem 1.2]{GJNT16} we have that if $E$ gains a point $P_3$ of order 3 over a sextic field then $[\Q(P_3):\Q]=6$ and there is no other point of order $3$ defined over a number field of degree less than $6$. But this is impossible by Table \ref{TABLE}.\\

(\ref{2t24}) Suppose $G = (2), (4)$ or $(6)$. In all of these cases $E$ has a rational point of order $2$ and so the Tate module $T_2(E)$ is a tower of $2$-extensions. Therefore, if $P_8$ is a point of order 8 on $E$ defined over a sextic extension of $\Q$, then it must be that $[\Q(P_8):\Q] = 2$ and thus, by \cite[Lemma 4.6]{Q3},  $E$ has a rational $8$-isogeny. Further, if $E$ gains a point of order 3 over a sextic extension of $\Q$, then by Lemma \ref{lemIsog} $E$ must have a rational 3-isogeny. Therefore in all of these cases, if $H=(24)$ then $E$ must have  a rational 24-isogeny which is impossible by  Theorem \ref{isogQ}. \\

(\ref{t26}) Let be $P_2,P_{13}\in E(K)_{\tors}$ of order 2 and 13 respectively. By Lemma \ref{lemIsog}, $E$ has a rational $13$-isogeny. In the case that $\Q(P_2)=\Q$ we have that there is a rational $26$-isogeny which cannot happen by Theorem \ref{isogQ}. 
Now, if $\Q(P_2)\ne\Q$, then $[\Q(P_{2}):\Q]=3$ and $\Q(P_2)\subseteq \Q(P_{13})$. Note that $\Q(P_2)\ne \Q(P_{13})$, since $(26)$ is not a subgroup of some group in $\Phi_\Q(3)$. 
Therefore $[\Q(P_{13}):\Q]=6$ and Table \ref{TABLE} shows that $G_E(13)$ is conjugate to \texttt{13B.3.4} or \texttt{13B.4.1}. 
In both cases we have that the field $\Q(P_{13})$ is Galois and cyclic of order $6$ (see (2) from Section \ref{sec:aux}). 
If $\rho_{E,2}$ is non-surjective, then the Galois group of $\Q(P_2)$ is isomorphic to $\GL_2(\Z/2 \Z)$ contradicting the fact that $\Q(P_2)\subsetneq \Q(P_{13})$. The remaining case is when $G_E(2)$ is conjugate to \texttt{2Cn}. The fiber product of the genus $0$ modular curves $X_0(13)$ and $X_{\texttt{2Cn}}$ is the elliptic curve with Cremona label \href{http://www.lmfdb.org/EllipticCurve/Q/52a2}{\texttt{52a2}} who has only the affine point $(0,0)$. This point does not correspond to an elliptic curve in $X_0(13)$ or $X_{\texttt{2Cn}}$. Therefore we have proved that $(26)\not\subseteq E(K)_{\tors}$. \\

(\ref{t27}) Suppose that $H = (27)$, then $G = (1), (3)$ or $(9)$. By Lemma \ref{lemIsog}, $E$ have a rational 3-isogeny. Therefore, $G_E(27)$ must conjugate to a subgroup of $\pi^{-1}(B(3))$, where $B(3)$ is the Borel subgroup of $\GL_2(\Z/3\Z)$ and $\pi\colon \GL_2(\Z/27\Z) \to \GL_2(\Z/3\Z)$. Constructing a list of the subgroups of $\pi^{-1}(B(3))$ with surjective determinant, containing an element corresponding to complex conjugation, and  not conjugate to a subgroup of $B(27)$ (since the only curves with a rational 27-isogeny have CM) we find there are $687$ of possible images for $\rho_{E,27}$. Of those $687$ of images, $42$ of them would give rise to an elliptic curve with a point of order $27$ over a sextic. Among those $42$ possibilities there are exactly $7$ maximal elements all of which are conjugate to a subgroup of $B(9)$. Five of these $7$ maximal groups reduce to subgroups of $\GL_2(\Z/3\Z)$ that are conjugate to subgroups of the split Cartan subgroup. {If $E$ were an elliptic curve whose mod $27$ Galois representation is conjugate to a subgroup of one of these $5$ groups, then $E$ would have to have independent rational $3$- and $9$-isogenies. From \cite[Lemma 7]{N15a}, any elliptic curve with independent rational $3$- and $9$-isogenies is isogenous to an elliptic curve with a rational 27-isogeny and from \cite[Table 4]{Lozano} there is only one elliptic curve (up to $\Qbar$-isomorphism) with a rational 27-isogeny. The $\Qbar$-isomorphism class of curves with a 27-isogeny has $j$-invariant equal to $-2^{15}\cdot3\cdot5^3$ and consist of all quadratic twists of the elliptic curve with Cremona Reference \href{http://www.lmfdb.org/EllipticCurve/Q/27a2}{\texttt{27a2}}. Since isogeny classes are invariant under quadratic twist it is sufficient to check \href{http://www.lmfdb.org/EllipticCurve/Q/27/a/}{the isogeny class of} \href{http://www.lmfdb.org/EllipticCurve/Q/27a2}{\texttt{27a2}} in the LMFDB database \cite{lmfdb}} to see the only way that $E$ can have an independent rational 3- and 9-isogeny is if $j(E)=0$ and $E$ thus has CM. Therefore, $G_E(27)$ must be conjugate to a subgroup of one of the remaining two maximal groups. These two groups are 
\[
G_1 =  \left\langle
\begin{pmatrix} 1&0\\0&-1 \end{pmatrix}, \begin{pmatrix} 2&0\\0&1 \end{pmatrix}, \begin{pmatrix} 10&22\\18&10 \end{pmatrix}
 \right\rangle 
\hbox{  and  }
G_2 =  \left\langle
\begin{pmatrix} 1&0\\0&-1 \end{pmatrix}, \begin{pmatrix} 2&0\\0&1 \end{pmatrix}, \begin{pmatrix} 19&4\\9&19 \end{pmatrix}
 \right\rangle.
\]
Working in Magma we see that $G_1\cap \SL_2(\Z/27\Z)$ (resp.~$G_2\cap \SL_2(\Z/27\Z)$) is conjugate to a subgroup of the group with Cummins-Pauli label $27A^2$ (resp.~$27B^4$). The modular curves that correspond to the group $27A^2$ and $27B^4$ are genus 2 and 4 respectively and it was shown in the proof of Proposition 5.18 in \cite{Q3}, that the only $\Q$-rational points on these curves are cusps. Therefore there are no elliptic curves over $\Q$ whose mod $27$ image is contained in either $G_1$ or $G_2$. \\

(\ref{t28})  If $H = (28)$ then $E$ gain a point of order $7$ over a sextic field. Then, by Lemma \ref{lemIsog}, if $E$ has not a rational 7-isogeny then $E(K)_{\tors}=(7)$ or $E(K)_{\tors}=(2,2)$. Therefore we have that $E$ must have a rational $7$-isogeny. Now, first suppose that $(2)\subseteq G$, then $E$ would have a rational $14$-isogeny defined over $\Q$ and from Theorem \ref{isogQ} this can only happen if $E$ has complex multiplication. \\

If $G = (1)$ and $H = (28)$ then $E$ must gain a point of order $4$ over a cubic or a sextic extension of $\Q$ which from part (\ref{t2}) can only happen if $E$ corresponds to a rational point on the genus $0$ curve $X_{20}$ from \cite{RZB}. So it must be that $E$ comes from a point on $X_{20}$ and a point on $X_0(7)$. Computing the fiber product of $X_0(7)$ and $X_{20}$, we get a genus 3 hyperelliptic curve $C$ with $\Aut(C)= (2,2,2)$. Quotienting out by one of the automorphisms of order 2 that is not the hyperelliptic involution we get the elliptic curve $E'$ with Cremona label \href{http://www.lmfdb.org/EllipticCurve/Q/14a4}{\texttt{14a4}} satisfying $E'(\Q)=(6)$. Pulling the 6 points in $E'(\Q)$ back to $C(\Q)$ we see that there are exactly 4 non-cuspidal and non-singular rational points on $C$ corresponding to the $j$-invariants $-3^3\cdot13\cdot479^3/2^{14}$ and $ 3^3\cdot 13/2^2$. 
From \cite[Theorem 1.5]{zywina1}, all of the twists of these curves have the kernel of their rational 7-isogeny defined over a cyclic sextic extension of $\Q$ except for two of them. Further, the only way that $E$ can have the 2-primary component of its torsion grows from trivial to $(4)$ over a cyclic sextic extension of $\Q$ is if it actually grow over a quadratic extension of $\Q$. This is because the point of order 4 would define the kernel of a rational isogeny and hence by \cite[Lemma 4.8]{Q3} the degree of its field of definition would have to divide $\varphi(4) = 2$, but this cannot happen by \cite[Theorem 2]{GJT14}. Therefore, we can rule out these curves since $\Q(P_7)$ can never coincide with $\Q(P_4)$. Next from \cite[Theorem 1.5]{zywina1}, we know that for each of these $j$-invariants there are exactly one twist (up to $\Q$-isomorphism) such that the kernel of their rational 7-isogeny are defined over a cyclic cubic extensions of $\Q$.  The $\Q$-isomorphism classes in question are represented by the elliptic curves with Cremona label \href{http://www.lmfdb.org/EllipticCurve/Q/338b1}{\texttt{338b1}} and \href{http://www.lmfdb.org/EllipticCurve/Q/16562be2}{\texttt{16562be2}} and we eliminate these by checking that the 2-division field and the cubic field where the kernel of the rational 7-isogeny intersect only in $\Q$. \\

Next if $G=(7)$, again E must have a rational 7-isogeny and so in order to gain a point of order 4 it must have $j$-invariant  $-3^3\cdot13\cdot479^3/2^{14}$ and $ 3^3\cdot 13/2^2$, but none of these curves have a rational point of order 7. \\

(\ref{t30}) By Lemma \ref{lemIsog} the only way for a curve to gain a point of order 3 or 5 over a sextic extension is for $E$ to have a rational 3- or 5-isogeny respectively. Therefore, if $G=(2), (6)$ or $(10)$ and $H=(30)$, $E$ must have rational 30-isogeny which is impossible by Theorem \ref{isogQ}. 

If $G =(1)$, then again $E$ must have a rational 15-isogeny and $j(E) \in \{ -5^2/2, -5^2\cdot 241^3/2^3, -5^2\cdot29^3/2^5, 5\cdot211^3/2^{15} \}$ (see \cite[Table 4]{Lozano}). Further,  by \cite[Lemma 4.8]{Q3} we know that the kernel of the rational 15-isogeny is defined over a field of degree dividing $\varphi(15) = 8$. Therefore, if the kernel of the rational 15-isogeny is defined over a sextic field, it must be in fact be defined over a quadratic field. But from \cite[Theorem 2 (c)]{N15a} there are no curves with $G=(1)$ and a point of order 15 over a quadratic field, so this is not possible. Thus the kernel of the rational 15-isogeny cannot be defined over a sextic extension. 

Looking at Table \ref{TABLE} we see that the only way to have a point of order 3 and a point of order 5 defined over a sextic field without having any rational points (since we are in the case where $G = (1)$) is for the point of order 5 to be defined over a quadratic extension of $\Q$ (i.e.~$G_E(5)$ is conjugate to a subgroup of \texttt{5B.4.1}) and the point of order 3 has to be defined over a cubic extension or a sextic extension of $\Q$ (i.e.~$G_E(3)$ is conjugate to \texttt{3B.2.1} or \texttt{3B} respectively). We point out here that \texttt{3B.2.1} is contained inside of \texttt{3B} and in fact the group generated by \texttt{3B.2.1} and $-I$ is equal \texttt{3B}. Therefore the $j$-maps from these modular curves are the same and for each $\Qbar$-isomorphism class corresponding to a point on the modular curve  $X_\texttt{3B} = X_0(3)$ there is a unique twist such that $G_E(3)$ is conjugate to \texttt{3B.2.1}. Constructing the fiber product of these two genus 0 modular curves we get the elliptic curve $E'/\Q$ with Cremona label \href{http://www.lmfdb.org/EllipticCurve/Q/15a3}{\texttt{15a3}} and $E'(\Q)=(2,4)$. The rational points of $E'$ give 4 nonsingular noncuspidal points corresponding to the two j-invariants $5\cdot 7\cdot 11\cdot 43\cdot 421/2^{15}$ and $-5\cdot 29^3/2^5$. Using \cite[Theorem 1.2]{zywina1}, we see that each of these two curves has exactly 1 twist (up to $\Q$-isomorphism) with a point of order 3 defined over a cubic field. The $\Q$-isomorphism classes in question are represented by curves with Cremona label \href{http://www.lmfdb.org/EllipticCurve/Q/50a4}{\texttt{50a4}} and \href{http://www.lmfdb.org/EllipticCurve/Q/450b3}{\texttt{450b3}}. Examining the division fields of these curves we see that the cubic fields where they gain a point of order 3 are disjoint from the 2-division fields. Since these division fields are invariant under twisting, no twist of these curves can gain a point of order 30 over a sextic extension of $\Q$.

Before moving on, we point out the following: if $E/\Q$ is either the elliptic curve with Cremona label \href{http://www.lmfdb.org/EllipticCurve/Q/50a4}{\texttt{50a4}} or \href{http://www.lmfdb.org/EllipticCurve/Q/450b3}{\texttt{450b3}} and $P_{15}$ is the point of order 15 defined over a sextic extension of $\Q$, then from \cite[Theorems 1.2 and 1.4]{zywina1} twisting $E$ by a square free integer $d$ can only affect the degree of the field of definition of $P_{15}$ in the following ways. If $d = -3$, then $E^d$ has a rational point of order 3 and if $d = 5$ then $E^d$ has a rational point of order 5 and in both of these cases $G\neq (1)$. If $d$ is any other square-free integers, the field of definition of the point of order 15 becomes $\Q(P_{15},\sqrt{d})/\Q$ which is a degree 12 extension. Therefore, these are the only two curves with $G = (1)$ and $H=(15)$ up to $\Q$-isomorphism.
\end{proof}

For the case when $E/\Q$ is an elliptic curve with complex multiplication give the following result:

\begin{proposition}\label{teoCM}
Let $E/\Q$ be an elliptic curve with complex multiplication and $K/\Q$ a sextic number field such that $E(\Q)_{\tors}=G$ and $E(K)_{\tors}= H$.
\begin{alphenumCM}
\item \label{CMt1} $11,13, 17$ and $19$ do not divide the order of $H$.
\item\label{CMt2x4} If $G=(1)$ or $(2),$ then $H\ne (2,4)$.
\item\label{CMt2x14} If $G=(2,2)$, then $(2,14)\not\subseteq H$.
\item\label{CMt3x6} If $G = (1)$, then $H\ne (3,6)$.
\item\label{CMt18} If $G = (3),$ then $H\ne (18)$.
\end{alphenumCM}
\end{proposition}

\begin{proof}
(\ref{CMt1}) We know that $R_\Q(6)=\{2,3,5,7,13\}$, then we only need to prove the statement for $13$. By \cite[\S 1.8]{zywina1} we have that $G_E(13)$ is \texttt{13Ns}, \texttt{13Nn} or $G^3(13)$. Therefore by Theorem 5.6 \cite{GJN16} we have an explicit characterization of the degree $[\Q(P_{13}):\Q]$ where $P_{13}$ is a point of order $13$. In particular the minimum degree of $\Q(P_{13})/\Q$ is $24$.\\

(\ref{CMt2x4}) First suppose towards a contradiction that $G=(1)$ and $H=(2,4)$. From \cite{GJT14, GJNT16} we know that this growth cannot happen over a quadratic or cubic extension of $\Q$. Next since $G = (1)$ we have that $[\Q(E[2]):\Q] =3$ or $6$ and the only way that $[\Q(E[2]):\Q] = 3$ is if the discriminant of $E$ is a square. Looking at the tables of CM elliptic curves over $\Q$ in \cite[Appendix A, \S 3]{SilvermanAdTopics} we check that there is only one $\Qbar$-isomorphism class of CM elliptic curves with an square discriminant, all of the form $y^2 = x^3-r^2x$ with $r\in\Q$. All of these curves have full $2$-torsion over $\Q$ and thus do not have $G=(1)$. So we may assume that $[\Q(E[2]):\Q] = 6$ and thus $\Q(E[2]) = K$ is a Galois extension of $\Q$. Now, from \cite[Lemma 4.6]{Q3} since $K/\Q$ is Galois and $E(K)_{\tors}= (2,4)$ we know that $E$ must have a rational 2-isogeny which can only happen if $E$ has a rational point of order 2 contradicting the assumption that $G=(1)$. 

Next suppose towards a contradiction that $G=(2)$ and $H=(2,4)$. Again, from \cite{GJT14, GJNT16} we know that this growth cannot happen over a quadratic or cubic extension of $\Q$. In this case $[\Q(E[2]):\Q]=2$ and so there must be a point of order 2 defined over a quadratic field that become divisible by 2 in a cubic extension, but this is impossible by \cite[Proposition 4.6]{GJN16}.\\

(\ref{CMt2x14}) Since $G=(2,2)$ we have that $E:y^2=x^3-r^2x$ for some $r\in \Q$ (see \cite[Proposition 1.15]{zywina1}). In particular $j(E)=1728$ and \cite[Proposition 1.14]{zywina1} shows that $G_E(7)$ is conjugate to \texttt{7Nn}. In this case, if $P$ is a point of order $7$, we have $[\Q(P):\Q]=48>6$.\\

(\ref{CMt3x6}) Analogous to the proof of Proposition \ref{teo} (\ref{t6x6}) we need a 
point $P_2\in E[2]$ not defined over $\Q$, no rational points of order 
$3$ and $[\Q(E[3]):\Q]$ divides $6$. By Propositions 1.14, 1.15 and 1.16 
from \cite{zywina1} and \cite[Theorem 5.7]{GJN16} this is only possible
if $\rho_{E,2}$ is non-surjective and $G_E(3)$ is conjugate to \texttt{3B.1.2}. This gives the same 
contradiction as Proposition \ref{teo} (\ref{t6x6}).\\

(\ref{CMt18}) Note that in the proof of the similar statement but without complex multiplication (see Proposition \ref{teo} (\ref{t18})) we have not used the condition that the elliptic curve is without complex multiplication. 
\end{proof}

\begin{theorem}\label{c21}
Let $E/\Q$ be an elliptic curve. Suppose $E(K)_{\tors}=(21)$  over some sextic number field $K$. Then $j(E) \in\{3^3\cdot 5^3/2,-3^2\cdot5^3\cdot 101^3/2^{21},-3^3\cdot 5^3\cdot 382^3/2^7,-3^2\cdot 5^6/2^3\}$. 
\end{theorem}

\begin{proof}
By Lemma \ref{lemIsog}, If $E$ has Cremona label \href{http://www.lmfdb.org/EllipticCurve/Q/2450ba1}{\texttt{2450ba1}} or \href{http://www.lmfdb.org/EllipticCurve/Q/2450bd1}{\texttt{2450bd1}} then $E(K)_{\tors}\ne (21)$, otherwise if $E(K)_{\tors}=(21)$ then $E$ has a rational $3-$ and a rational $7-$isogeny. Therefore a rational $21$-isogeny. That is $j(E) \in \{3^3\cdot 5^3/2,-3^2\cdot5^3\cdot 101^3/2^{21},-3^3\cdot 5^3\cdot 382^3/2^7,-3^2\cdot 5^6/2^3\}$ by the classification of the rational points in $X_0(21)$ (see \cite[Table 4]{Lozano}). 
\end{proof}

\begin{theorem}\label{c1530} Let $E/\Q$ be an elliptic curve, $K/\Q$ a sextic number field and $E(K)_{\tors}= H$. Then
\begin{romanenum}
\item If $H=(15)$, then $E$ has Cremona label \href{http://www.lmfdb.org/EllipticCurve/Q/50a3}{\texttt{50a3}}, \href{http://www.lmfdb.org/EllipticCurve/Q/50a4}{\texttt{50a4}}, \href{http://www.lmfdb.org/EllipticCurve/Q/50b1}{\texttt{50b1}}, \href{http://www.lmfdb.org/EllipticCurve/Q/50b2}{\texttt{50b2}}, \href{http://www.lmfdb.org/EllipticCurve/Q/450b4}{\texttt{450b4}}, or \href{http://www.lmfdb.org/EllipticCurve/Q/450b3}{\texttt{450b3}}.
\item If $H=(30)$, then $E$ has Cremona label \href{http://www.lmfdb.org/EllipticCurve/Q/50a3}{\texttt{50a3}}, \href{http://www.lmfdb.org/EllipticCurve/Q/50b1}{\texttt{50b1}}, \href{http://www.lmfdb.org/EllipticCurve/Q/50b2}{\texttt{50b2}}, or \href{http://www.lmfdb.org/EllipticCurve/Q/450b4}{\texttt{450b4}}.
\end{romanenum}
\end{theorem}
\begin{proof}
From the proof of Proposition \ref{teo} (\ref{t30}), we know that the only elliptic curves with $G=(1)$ and $H=(15)$ or $(30)$ are the elliptic curves with \href{http://www.lmfdb.org/EllipticCurve/Q/50a4}{\texttt{50a4}} and \href{http://www.lmfdb.org/EllipticCurve/Q/450b3}{\texttt{450b3}} and they both have $H = (15)$. Therefore, all that remains to classify are the elliptic curve with $G=(3)$ or $G=(5)$ and $H = (15)$ or $H=(30)$. 

Starting with an elliptic curve with $G = (3)$, from Table \ref{TABLE} we get that the only way to gain a point of order 5 over a sextic field is to gain it over a quadratic extension of $\Q$. So first classify all elliptic curves with a point of order 3 over $\Q$ and a rational 5-isogeny defined over a quadratic extension. In order to have a rational point of order 3, it must be that $G_E(3)$ is conjugate to a subgroup of \texttt{3B.1.1} while having a point of order 5 defined over a quadratic extension requires $G_E(5)$ to be conjugate to a subgroup of \texttt{5B.4.1}. Computing the fiber product of the two corresponding genus 0 modular curves we see that there are exactly two curves (up to $\Q$-isomorphism) that simultaneously have these properties and they are ones with Cremona label \href{http://www.lmfdb.org/EllipticCurve/Q/450b4}{\texttt{450b4}} and \href{http://www.lmfdb.org/EllipticCurve/Q/50a3}{\texttt{50a3}}. Because these curves gain a point of order 15 over a quadratic extension and every elliptic curve with trivial 2 torsion gains a point of order 2 over a cubic extension, we know that for each of \href{http://www.lmfdb.org/EllipticCurve/Q/450b4}{\texttt{450b4}} and \href{http://www.lmfdb.org/EllipticCurve/Q/50a3}{\texttt{50a3}} there is a sextic extension of $\Q$ where $H = (15)$ and a sextic extension of $\Q$ where $H = (30)$.

Lastly, we classify the elliptic curves that have $G = (5)$ and a rational 3-isogeny. This time to have a rational point of order 5, $G_E(5)$ must be conjugate to \texttt{5B.1.1} while a rational 3-isogeny requires that $G_E(3)$ is conjugate to a subgroup of $\texttt{3B}$. Computing the fiber product of these two genus 0 modular curves we see that there are exactly two elliptic curves (up to $\Q$-isomorphism) with both these properties simultaneously and they are the curves with Cremona label \href{http://www.lmfdb.org/EllipticCurve/Q/50b1}{\texttt{50b1}} and \href{http://www.lmfdb.org/EllipticCurve/Q/50b2}{\texttt{50b2}}. Again for both \href{http://www.lmfdb.org/EllipticCurve/Q/50b1}{\texttt{50b1}} and \href{http://www.lmfdb.org/EllipticCurve/Q/50b2}{\texttt{50b2}} there is a quadratic extension of $\Q$ where $E$ gains a point of order 3 and so for each of these curves there is a sextic extension of $\Q$ where $H = (15)$ and a sextic extension of $\Q$ where $H = (30)$.
\end{proof}

\begin{remark} The following table shows the sextic fields (or subfields) where the torsion grows to $(15)$ or $(30)$:
\begin{center}
\begin{tabular}{|c|c|c|c|}
\hline
$G$ & \backslashbox{$E$}{$H$} & $(15)$ & $(30)$\\
\hline 
$(3)$ & \href{http://www.lmfdb.org/EllipticCurve/Q/50a3}{\texttt{50a3}} & $\Q(\sqrt{5})$ &  $\Q(\sqrt{5},\alpha)$ \\
\hline
$(1)$ & \href{http://www.lmfdb.org/EllipticCurve/Q/50a4}{\texttt{50a4}} & $\Q(\sqrt[6]{5})$ & $-$ \\
\hline
$(5)$ & \href{http://www.lmfdb.org/EllipticCurve/Q/50b1}{\texttt{50b1}} &  $\Q(\sqrt{5})\,,\, \Q(\sqrt{-15},\beta)$ &  $\Q(\sqrt{5},\alpha)$ \\
\hline
$(5)$ & \href{http://www.lmfdb.org/EllipticCurve/Q/50b2}{\texttt{50b2}} & $\Q(\sqrt{-15})\,,\, \Q(\sqrt[6]{5})$ & $\Q(\sqrt{-15},\alpha)$\\
\hline 
$(3)$ & \href{http://www.lmfdb.org/EllipticCurve/Q/450b4}{\texttt{450b4}} & $\Q(\sqrt{-15})$ &  $\Q(\sqrt{-15},\alpha)$\\
\hline
$(1)$ &  \href{http://www.lmfdb.org/EllipticCurve/Q/450b3}{\texttt{450b3}} & $\Q(\sqrt{-15},\beta)$ & $-$ \\
\hline
\end{tabular}
\end{center}
where $\alpha^3-\alpha^2+2\alpha+2=0$ and $\beta^3-\beta^2-3\beta-3=0$.
\end{remark}

\begin{theorem}\label{families}
There are infinitely many non-isomorphic (over $\overline{\Q}$) elliptic curves $E/\Q$ such that there is a sextic number field $K$ with $E(K)_{\tors}=(2,18)$ (resp.~$(3,9)$, $(3,12)$, $(6,6)$).
\end{theorem}
\begin{proof}
Given $G\in \Phi(1)$ there exist a one-parameter family, called the Kubert--Tate normal form, 
$$
{\mathcal T}^G_{t}:\ y^2+(1-c)xy-by=x^3-bx^2,\qquad \mbox{where $b,c\in\Q(t)$},
$$
such that $G$ is a subgroup of ${\mathcal T}^G_{t}(\Q(t))_{\tors}$ for all but finitely many values of $t\in\Q$ (cf.  \cite[Table 3]{Kubert}). For $G=(9)$ and $G=(12)$ we have
$$
\begin{array}{lclclc}
  G=(9) & : &  c=t^2(t-1) & , &  b=c(t^2-t+1)& \quad t\ne 0,1,\\
  G=(12) & : &  c=(3t^2-3t+1)(t-2t^2)/(t-1)^3 & , & b=c(2t-2t^2-1)/(t-1),& \quad t\ne 0,1,1/2
\end{array}
$$
Thanks to the classification of $\Phi(1)$ we have that for those values of $t\in\Q$ we have ${\mathcal T}^G_{t}(\Q)_{{\tors}}=G$. Moreover, ${\mathcal T}^G_{t}$ has not CM since $G\notin \Phi^{\text{CM}}(1)$. Since ${\mathcal T}^G_{t}$ has a rational point of order $3$, Table \ref{TABLE} tell us that the image of the mod $3$ Galois representation attached to ${\mathcal T}^G_{t}$ is labeled \texttt{3Cs.1.1} or \texttt{3B.1.1}. The former case can not happen since in that case $[\Q({\mathcal T}^G_{t}[3]):\Q]=2$, but $(3,9),(3,12)$  are not subgroups of some group in $\Phi_{\Q}(2)$. For the case \texttt{3B.1.1} we have $[\Q({\mathcal T}^G_{t}[3]:\Q]=6$. Therefore $K=\Q({\mathcal T}^G_{t}[3])$ is a sextic field satisfying $(3,9)$ if $G=(9)$ (resp.~ $(3,12)$ if $G=(12)$) is a subgroup of ${\mathcal T}^G_{t}(K)$. But this happens for infinitely many values of $t$ and since $(3,9)$ (resp.~ $(3,12)$) is a maximal subgroup in $\Phi^\infty(6)$ we have that, in fact, ${\mathcal T}^G_{t}(K)_{\tors}=(3,9)$ (resp $=(3,12)$).

Now, let be $G=(9)$ and $K=\Q({\mathcal T}^G_{t}[2])$. Then $(2,18)$ is a subgroup of ${\mathcal T}^G_{t}(K)$. We have that $[K:\Q]=6$, since otherwise $(2,18)$  is a subgroups of some group in $\Phi_\Q(d)$ for $d$ dividing $6$. That is impossible. Analogous to the cases above, we have that in this case  ${\mathcal T}^G_{t}(K)_{\tors}=(2,18)$.

Finally the case $(6,6)$. Let be the one-parameter family\footnote{This family was computed by the first author and \'Alvaro Lozano-Robledo while working on the question of when can $\Q(E[n]) = \Q(E[m])$. The family $\mathcal{A}_{t}$ is the 1-parameter family of elliptic curves with the property that $\Q(\mathcal{A}_{t}[2]) = \Q(\mathcal{A}_{t}[3])$ and both are sextic extensions of $\Q$.} given by:
$$
\mathcal{A}_{t}\,:\,y^2=x^3-3(a-1)^3(a-9)x-2(a-1)^4(a^2 + 18a - 27),\qquad a=(t^3-1)^2,\qquad t\ne 0,1.
$$
This elliptic curve satisfies $\Q(\mathcal{A}_{t}[2])=\Q(\mathcal{A}_{t}[3])$. In particular, the field $K=\Q(\mathcal{A}_{t}[6])$ is of degree $6$. Then $(6,6)$ is a subgroup of $\mathcal{A}_{t}(K)$. An analogous argument to above proves that $\mathcal{A}_{t}(K)_{\tors}=(6,6)$.

Since the $j$-invariant of $\mathcal T_{t}^G$ and $\mathcal{A}_{t}$ is not constant, this proves that there are infinitely many non $\overline{\Q}$-isomorphic elliptic curve over $\Q$ with torsion structures $H=(2,18)$, $(3,9)$, $(3,12)$ or $(6,6)$ over sextic fields. 
\end{proof} 

\section{Proof of Theorems \ref{main1} and \ref{main2}}
The results in Section \ref{sec_aux} are exactly the results necessary to proof the main theorems of the paper.

\begin{proof}[Proof of Theorem \ref{main1}]
For $d=2,3$ we have that $\Phi_\Q(d)\subseteq \Phi_\Q(6)$ and $\Phi_\Q(d)\subseteq \Phi^\infty(6)$. Therefore $\Phi_\Q(2)\cup \Phi_\Q(3)\subseteq \Phi^\star_\Q(6)$. For $H\in \{(30),(2,18),(3,9),(3,12),(6,6)\}$ we have examples at Table \ref{ex_6} of an elliptic curve $E/\Q$, a sextic number field $K$ such that $E(K)_{\tors} = H$. Now, by definition, $\Phi_\Q^\star(6)\subseteq \Phi^\infty(6)$, so our task to complete the description of $\Phi_\Q^\star(6)$ is to prove that the following torsion structures do no appear for elliptic curve over $\Q$ base change to any sextic number field:
$$
(11),\,(17),\,(19),\,(20),\,(22),\,(24),\,(26),\,(27),\,(28),\,(2,16),\,(2,20),\,(4,8).
$$
Indeed, 
\begin{itemize}
\item $H\ne (11),(17),(19),(22)$ by Proposition \ref{teo} (\ref{t1}) and by Proposition \ref{teoCM} (\ref{CMt1}).
\item $H\ne (20)$ by Proposition \ref{teo}  (\ref{t20}).
\item $H\ne (24)$ by Proposition \ref{teo} (\ref{t2}), (\ref{t6quartic}),  (\ref{t24}) and (\ref{2t24}). 
\item $H\ne (26)$ by Proposition \ref{teo} (\ref{t26}) and by Proposition \ref{teoCM} (\ref{CMt1}).
\item $H\ne (27)$ by Proposition \ref{teo}  (\ref{t27}).
\item $H\ne (28)$ by Proposition \ref{teo} (\ref{t28}).
\item $H\ne (2,16), (4,8)$ by Proposition \ref{teo} (\ref{t2}).
 \item $H\ne (2,20)$ by Proposition \ref{teo} (\ref{t2}), (\ref{t7quartic}) and (\ref{t5quartic}).
\end{itemize}
This concludes the first part of Theorem \ref{main1}, that is, the determination of $\Phi_\Q^\star(6)$. In particular we have obtained
\begin{equation}\label{inf6}
\Phi_\Q^\star(6)=\Phi_\Q(2)\cup\Phi_\Q(3)\cup\{(30),(2,18),(3,9),(3,12),(6,6)\}.
\end{equation}
Now part (i) comes from Theorem \ref{c21} and (ii), (iii) from Theorem \ref{c1530}. It remains to determine $\Phi_\Q^\infty(6)$. Najman \cite{N15a} has proved that $\Phi_\Q^\infty(2)=\Phi_\Q(2)\setminus \{(15)\}$ and $\Phi_\Q^\infty(3)=\Phi_\Q(3)\setminus \{(21)\}$. Then by (\ref{inf6}) and (i-iii) we only need to prove that there are infinitely many non $\Qbar$-isomorphic classes of elliptic curves over $\Q$ such that base change to sextic number field the torsion grows to one of the group in $\{(2,18),(3,9), (3,12), (6,6)\}$. This have been done in Theorem \ref{families}.
\end{proof}

\begin{proof}[Proof of Theorem \ref{main2}]
The groups $H\in\Phi^{\star}_\Q(6)$ that do not appear in some $\Phi^{\star}_\Q(6,G)$ for any $G\in\Phi(1)$, with $G\subseteq H$,  can be ruled out using Propositions \ref{teo} and \ref{teoCM}. In Table \ref{table1} below, for each group $G$ at the top of a column, we indicate what groups $H$ (in each row) may appear, and indicate
\begin{itemize}
\item with (\ref{t1})--(\ref{t30}), which part of Proposition \ref{teo} is used to prove that the pair $(G,H)$ cannot appear in the non-CM case,
\item with (\ref{CMt1})--(\ref{CMt18}), which part of Proposition \ref{teoCM} is used to prove that the pair $(G,H)$ cannot appear in the CM case,
 \item with $-$, if the case is ruled out because $G \not\subseteq H$, 
\item with a $\checkmark$, if the case is possible and, in fact, it occurs. There are a few types of check marks in Table \ref{table1}:
\begin{itemize}  
\item $\checkmark$ (without a subindex) means\footnote{Note that for any positive integer $d$, and any elliptic curve $E/\Q$ with $E(\Q)_{\tors}= G$, there is always an extension $K/\Q$ of degree $d$ such that $E(K)_{\tors}= E(\Q)_{\tors}$ (and, in fact, this is the case for almost all degree $d$ extensions).} that $G=H$.
\item $\checkmark_{\!\!\!d}$ for $d=2$ or $3$ means that the structure $H$ occurs already over a quadratic\footnote{Examples can be found at Table 2 of \cite{GJT14}.} or cubic\footnote{Examples can be found at Table 1 of \cite{GJNT16}.} field respectively. $\checkmark_{\!\!\!2,3}$ means that both cases appear (not necessarily for a single elliptic curve).
\item  $\checkmark_{\!\!\!6}$ means that $H$ can be achieved over a sextic field but not over an intermediate quadratic or cubic field, and we have collected examples of curves and sextic fields in Table \ref{ex_6}.
\end{itemize}
\end{itemize}
\end{proof}

\begin{table}
{\footnotesize
\renewcommand{\arraystretch}{1.2}
\begin{longtable}[h]{|c|c|c|c|c|c|c|c|c|c|c|c|c|c|c|c|}
\hline
\backslashbox{$H$}{$G$}
 & $(1)$ & $(2)$ & $(3)$ & $(4)$ & $(5)$ & $(6)$ & $(7)$ & $(8)$ & $(9)$ & $(10)$ & $(12)$ & $(2,2)$ & $(2,4)$& $(2,6)$& $(2,8)$\\
\hline
\endfirsthead
\hline
\backslashbox{$H$}{$G$} 
& $(1)$ & $(2)$ & $(3)$ & $(4)$ & $(5)$ & $(6)$ & $(7)$ & $(8)$ & $(9)$ & $(10)$ & $(12)$ & $(2,2)$ & $(2,4)$& $(2,6)$& $(2,8)$\\
\hline
\endhead
\endfoot
\endlastfoot
\rowcolor{white} $(1)$ & $\checkmark$ & $-$ & $-$ & $-$ & $-$ & $-$ & $-$ & $-$ & $-$ & $-$ & $-$ & $-$ & $-$ & $-$ & $-$ \\
\hline
\rowcolor{white}  $(2)$ & $\checkmark_{\!\!\!3}$ & $\checkmark$ & $-$ & $-$ & $-$ & $-$ & $-$ & $-$ & $-$ & $-$ & $-$ & $-$ & $-$ & $-$ & $-$ \\
\hline
\rowcolor{white} $(3)$ & $\checkmark_{\!\!\!2,3}$ & $-$ & $\checkmark$ & $-$ & $-$ & $-$ & $-$ & $-$ & $-$ & $-$ & $-$ & $-$ & $-$ & $-$ & $-$ \\
\hline
\rowcolor{white} $(4)$ & $\checkmark_{\!\!\!3}$ & $\checkmark_{\!\!\!2}$ & $-$ & $\checkmark$ & $-$ & $-$ & $-$ & $-$ & $-$ & $-$ & $-$ & $-$ & $-$ & $-$ & $-$\\
\hline
\rowcolor{white} $(5)$ & $\checkmark_{\!\!\!2}$ & $-$ & $-$ & $-$ & $\checkmark$ & $-$ & $-$ & $-$ & $-$ & $-$ & $-$ & $-$ & $-$ & $-$ & $-$ \\
\hline
\rowcolor{white} $(6)$ &  $\checkmark_{\!\!\!3}$ & $\checkmark_{\!\!\!2,3}$ &  $\checkmark_{\!\!\!3}$ & $-$ & $-$ & $\checkmark$ & $-$ & $-$ & $-$ & $-$ & $-$ & $-$ & $-$ & $-$ & $-$\\
\hline
\rowcolor{white} $(7)$ & $\checkmark_{\!\!\!2,3}$ & $-$ & $-$ & $-$ & $-$ & $-$ & $\checkmark$ & $-$ & $-$ & $-$ & $-$& $-$ & $-$ & $-$ & $-$ \\
\hline
\rowcolor{white} $(8)$ &(\ref{t2}) & $\checkmark_{\!\!\!2}$ & $-$ & $\checkmark_{\!\!\!2}$ & $-$ & $-$ & $-$ & $\checkmark$ & $-$ & $-$ & $-$  & $-$ & $-$ & $-$ & $-$ \\
\hline
\rowcolor{white} $(9)$ & $\checkmark_{\!\!\!2}$ & $-$ & $\checkmark_{\!\!\!3}$ & $-$ & $-$ & $-$ & $-$ & $-$ & $\checkmark$ & $-$ & $-$& $-$ & $-$ & $-$ & $-$\\
\hline
\rowcolor{white} $(10)$ & $\checkmark_{\!\!\!6}$ & $\checkmark_{\!\!\!2}$ & $-$ & $-$ & $\checkmark_{\!\!\!3}$ & $-$ & $-$ & $-$ & $-$ & $\checkmark$ & $-$ & $-$ & $-$ & $-$ & $-$ \\
\hline
\rowcolor{light-gray} $(11)$ & (\ref{t1}) & $-$ & $-$ & $-$ & $-$ & $-$ & $-$ & $-$ & $-$ & $-$ & $-$ & $-$ & $-$ & $-$ & $-$ \\\hline
\rowcolor{white} $(12)$ & $\checkmark_{\!\!\!6}$ & $\checkmark_{\!\!\!2}$ & $\checkmark_{\!\!\!3}$ & $\checkmark_{\!\!\!2,3}$ & $-$ & $\checkmark_{\!\!\!2}$ & $-$ & $-$ & $-$ & $-$ & $\checkmark$ & $-$ & $-$ & $-$ & $-$\\
\hline
\rowcolor{white} $(13)$ & $\checkmark_{\!\!\!3}$& $-$ & $-$ & $-$ & $-$ & $-$ & $-$ & $-$ & $-$ & $-$ & $-$ & $-$ & $-$ & $-$ & $-$\\
\hline
\rowcolor{white} $(14)$ & $\checkmark_{\!\!\!6}$ & $\checkmark_{\!\!\!3}$ & $-$ & $-$ & $-$ & $-$ &  $\checkmark_{\!\!\!3}$ & $-$ & $-$ & $-$ & $-$ & $-$ & $-$ & $-$ & $-$\\
\hline
\rowcolor{white} $(15)$ & $\checkmark_{\!\!\!6}$ & $-$ & $\checkmark_{\!\!\!2}$ & $-$ & $\checkmark_{\!\!\!2}$ & $-$ & $-$ & $-$ & $-$ & $-$ & $-$ & $-$ & $-$ & $-$ & $-$\\
\hline
\rowcolor{white} $(16)$ &(\ref{t2}) & $\checkmark_{\!\!\!2}$ & $-$ & (\ref{t2}) & $-$ & $-$ & $-$ & $\checkmark_{\!\!\!2}$ & $-$ & $-$ & $-$  & $-$ & $-$ & $-$ & $-$\\
\hline
\rowcolor{light-gray} $(17)$ & (\ref{t1}) & $-$ & $-$ & $-$ & $-$ & $-$ & $-$ & $-$ & $-$ & $-$ & $-$  & $-$ & $-$ & $-$ & $-$\\
\hline
\rowcolor{white} $(18)$ & $\checkmark_{\!\!\!6}$ & $\checkmark_{\!\!\!6}$ & (\ref{t18}),(\ref{CMt18}) & $-$ & $-$ & $\checkmark_{\!\!\!3}$ & $-$ & $-$ & $\checkmark_{\!\!\!3}$ & $-$ & $-$  & $-$ & $-$ & $-$ & $-$\\
\hline
\rowcolor{light-gray} $(19)$ & (\ref{t1})(\ref{CMt1}) & $-$ & $-$ & $-$ & $-$ & $-$ & $-$ & $-$ & $-$ & $-$  & $-$  & $-$ & $-$ & $-$ & $-$\\
\hline
\rowcolor{light-gray}  $(20)$ & (\ref{t20}) & (\ref{t20}) & $-$ &  (\ref{t20})  & (\ref{t20})& $-$ & $-$ & $-$ & $-$ &  (\ref{t20}) & $-$  & $-$ & $-$ & $-$ & $-$\\
\hline
\rowcolor{white} $(21)$ & $\checkmark_{\!\!\!6}$ & $-$ & $\checkmark_{\!\!\!3}$ & $-$ & $-$ & $-$ & (\ref{t21}) & $-$ & $-$ & $-$ & $-$  & $-$ & $-$ & $-$ & $-$\\
\hline
\rowcolor{light-gray} $(22)$ & (\ref{t1})  & (\ref{t1}) & $-$ & $-$ & $-$ & $-$ & $-$ & $-$ & $-$ & $-$ & $-$  & $-$ & $-$ & $-$ & $-$\\
\hline
\rowcolor{light-gray}  $(24)$ & (\ref{t2}) & (\ref{2t24}) & (\ref{t2}) &(\ref{2t24}) & $-$ & (\ref{2t24}) & $-$ &  (\ref{t6quartic})  & $-$ & $-$ &(\ref{t24}) & $-$ & $-$ & $-$ & $-$\\
\hline
\rowcolor{light-gray}  $(26)$ & (\ref{t26}),(\ref{CMt1}) & (\ref{t26}),(\ref{CMt1})& $-$ & $-$ & $-$ & $-$ & $-$ & $-$ & $-$ & $-$ & $-$  & $-$ & $-$ & $-$ & $-$\\
\hline
\rowcolor{light-gray}  $(27)$ & (\ref{t27}) & $-$ & (\ref{t27}) & $-$ & $-$ & $-$ & $-$ & $-$ & (\ref{t27}) &  $-$ & $-$  & $-$ & $-$ & $-$ & $-$\\
\hline
\rowcolor{light-gray}  $(28)$ &(\ref{t28})&(\ref{t28}) & $-$ & (\ref{t28}) & $-$ & $-$ & (\ref{t28})& $-$ & $-$ & $-$ & $-$  & $-$ & $-$ & $-$ & $-$\\
\hline
\rowcolor{white}  $(30)$ & (\ref{t30}) & (\ref{t30}) &    $\checkmark_{\!\!\!6}$    & $-$ & $\checkmark_{\!\!\!6}$ & (\ref{t30})  & $-$ & $-$ & $-$ & (\ref{t30}) & $-$  & $-$ & $-$ & $-$ & $-$\\
\hline
\rowcolor{white} $(2,2)$ & $\checkmark_{\!\!\!3}$ & $\checkmark_{\!\!\!2}$ & $-$ & $-$ & $-$ & $-$ & $-$ & $-$ & $-$ & $-$ & $-$  & $\checkmark$ & $-$ & $-$ & $-$\\
\hline
\rowcolor{white} $(2,4)$ &(\ref{t2}),(\ref{CMt2x4}) &(\ref{t2}),(\ref{CMt2x4}) & $-$ & $\checkmark_{\!\!\!2}$ & $-$ & $-$ & $-$ & $-$ & $-$ & $-$ & $-$ & $\checkmark_{\!\!\!2}$ & $\checkmark$ & $-$ & $-$\\
\hline
\rowcolor{white} $(2,6)$ & $\checkmark_{\!\!\!6}$  & $\checkmark_{\!\!\!2}$ & $\checkmark_{\!\!\!3}$ & $-$ & $-$ & $\checkmark_{\!\!\!2}$ & $-$ & $-$ & $-$ & $-$ & $-$ & $\checkmark_{\!\!\!2,3}$ & $-$ & $\checkmark$ & $-$\\
\hline
\rowcolor{white} $(2,8)$ & (\ref{t2}) & (\ref{t2}) & $-$ & $\checkmark_{\!\!\!2}$ & $-$ & $-$ & $-$ & $\checkmark_{\!\!\!2}$ & $-$ & $-$ & $-$ &$\checkmark_{\!\!\!2}$ & $\checkmark_{\!\!\!2}$ & $-$ & $\checkmark$\\
\hline
\rowcolor{white} $(2,10)$ & $\checkmark_{\!\!\!6}$ & $\checkmark_{\!\!\!2}$ & $-$ & $-$ & $\checkmark_{\!\!\!6}$ & $-$ & $-$ & $-$ & $-$ & $\checkmark_{\!\!\!2}$ & $-$&  (\ref{t7quartic}) & $-$ & $-$ & $-$ \\
\hline
\rowcolor{white} $(2,12)$ & (\ref{t2}) & (\ref{t2}) & (\ref{t2})& $\checkmark_{\!\!\!2}$ & $-$ &  (\ref{t2}) & $-$ & $-$ & $-$ & $-$ & $\checkmark_{\!\!\!2}$ & $\checkmark_{\!\!\!2}$ & (\ref{t8quartic}) & $\checkmark_{\!\!\!2}$ & $-$ \\
\hline
\rowcolor{white}  $(2,14)$ & $\checkmark_{\!\!\!3}$ & $\checkmark_{\!\!\!6}$& $-$ & $-$ & $-$ & $-$ & $\checkmark_{\!\!\!6}$ & $-$ & $-$ & $-$ & $-$ & (\ref{t2x14}),(\ref{CMt2x14}) & $-$ & $-$ & $-$ \\
\hline
\rowcolor{light-gray}  $(2,16)$ & (\ref{t2}) & (\ref{t2})& $-$ & (\ref{t2}) & $-$ & $-$ & $-$ & (\ref{t2})  & $-$ & $-$ & $-$ &(\ref{t2}) & (\ref{t2}) & $-$ & (\ref{t2}) \\
\hline
\rowcolor{white} $(2,18)$ & $\checkmark_{\!\!\!6}$ & $\checkmark_{\!\!\!6}$ & (\ref{3t2x18}) & $-$ & $-$ & $\checkmark_{\!\!\!6}$ & $-$ & $-$ & $\checkmark_{\!\!\!6}$ & $-$ & $-$ & (\ref{t2x18}) & $-$ & (\ref{t2x18}) & $-$ \\
\hline
\rowcolor{light-gray}  $(2,20)$ & (\ref{t2}) & (\ref{t2}) & $-$ &  (\ref{t5quartic})  & (\ref{t2}) & $-$ & $-$ & $-$ &  $-$&(\ref{t2}) & $-$ & (\ref{t7quartic}) &  (\ref{t5quartic})  & $-$ & $-$ \\
\hline
\rowcolor{white} $(3,3)$ & $\checkmark_{\!\!\!6}$ & $-$ & $\checkmark_{\!\!\!2}$ & $-$ & $-$ & $-$ & $-$ & $-$ & $-$ & $-$ & $-$ & $-$ & $-$ & $-$ & $-$\\
\hline
\rowcolor{white}  $(3,6)$ & (\ref{t6x6}),(\ref{CMt3x6}) & $\checkmark_{\!\!\!6}$ & $\checkmark_{\!\!\!6}$ & $-$ & $-$ & $\checkmark_{\!\!\!2}$ & $-$ & $-$ & $-$ & $-$ & $-$& $-$ & $-$ & $-$ & $-$ \\
\hline
\rowcolor{white} $(3,9)$ &$\checkmark_{\!\!\!6}$  & $-$ & $\checkmark_{\!\!\!6}$ & $-$ & $-$ & $-$ & $-$ & $-$ & $\checkmark_{\!\!\!6}$ & $-$ & $-$& $-$ & $-$ & $-$ & $-$ \\
\hline
\rowcolor{white} $(3,12)$ & (\ref{t6x6}) & $\checkmark_{\!\!\!6}$ & (\ref{t3x12}) & $\checkmark_{\!\!\!6}$ & $-$ & $\checkmark_{\!\!\!6}$ & $-$ & $-$ & $-$ & $-$ & $\checkmark_{\!\!\!6}$ & $-$ & $-$ & $-$ & $-$ \\
\hline
\rowcolor{white} $(4,4)$ & $\checkmark_{\!\!\!6}$ & (\ref{t2}) & $-$ & $\checkmark_{\!\!\!2}$ & $-$ & $-$ & $-$ & $-$ & $-$ & $-$ & $-$ & (\ref{t2}) & $\checkmark_{\!\!\!2}$ & $-$ & $-$ \\
\hline
\rowcolor{light-gray}  $(4,8)$ & (\ref{t2}) & (\ref{t2}) & $-$ & (\ref{t2}) & $-$ & $-$ & $-$ & (\ref{t2}) & $-$ & $-$ & $-$ & (\ref{t2}) &(\ref{t2}) & $-$ & (\ref{t2}) \\
\hline
\rowcolor{white} $(6,6)$ & $\checkmark_{\!\!\!6}$ & $\checkmark_{\!\!\!6}$ & $\checkmark_{\!\!\!6}$ & $-$ & $-$ & $\checkmark_{\!\!\!6}$  & $-$ & $-$ & $-$ & $-$ & $-$ & $\checkmark_{\!\!\!6}$ & $-$ & $\checkmark_{\!\!\!6}$ & $-$ \\
\hline
\caption{The table displays either if the case happens for $G=H$ ($\checkmark$), if it already occurs over a quadratic field ($\checkmark_{\!\!\!2}$) or over a cubic field ($\checkmark_{\!\!\!3}$), if it occurs over a sextic but not a quadratic or cubic ($\checkmark_{\!\!\!6}$), if it is impossible because $G \not\subseteq H$ ($-$) or if it is ruled out by Proposition \ref{teo} ((\ref{t1})-(\ref{t30})) and Proposition \ref{teoCM} ((\ref{CMt1})-(\ref{CMt18}))}\label{table1}
\end{longtable}
}
\end{table}

\section{Sporadic torsion over sextic fields}\label{sporadic}
Let $d$ a positive integer and $H$ a finite abelian group. We say that $H$ is a sporadic torsion group for degree $d$ if there is a number field $K$ of degree $d$ and an elliptic curve $E/K$ such that $H=E(K)_{\tors}$ satisfying $H\notin \Phi^\infty (d)$. Notice that these elliptic curves are in bijection to its $j$-invariants in $J(d)$. To our knowledge only a few examples of sporadic torsion groups are known. Excluding Najman's elliptic curve mentioned in the introduction with $H=(21)$ and $d=3$, the known examples have been found by van Hoeij \cite{hoeij}:
\begin{itemize}
\item $d=5$: $H=(28),(30)$,
\item $d=6$: $H=(25),(37)$.
\end{itemize}
Interestingly, all of these the cases correspond to cyclic groups. In this paper we have dealt with the problem of characterizing $\Phi_\Q^\star(6)=\Phi_\Q(6)\cap \Phi^\infty(6)$ and it is easy to check that the two examples that van Hoeij found for $d=6$, namely $(25)$ and $(37)$, are not subgroups of any group in $\Phi_\Q(6)$. The former case because by Table \ref{TABLE} in order to gain a point $P_5$ of order $5$ in a sextic field we have $[\Q(P_5):\Q]=1,2$. Now if $P_{25}$ is a point of order $25$ such that $5P_{25}=P_5$ then by \cite[Proposition 4.6]{GJN16} we have $[\Q(P_{25}):\Q(P_5)]$ divides $25$ or $20$. Looking at $\Phi_\Q(d)$ for $d=2$ and $3$ we show that it is not possible the case $(25)$. The last case can be eliminated by simply noticing that $37\notin R_\Q(6)$.

On the other hand, in \cite{GJN18} the second author and Najman give two examples of elliptic curves over $\Q$ and a number field $K$ of degree 6 over $\Q$ such that $E(K)_{\tors} = (4,12)$ and thus showing that $(4,12)\in \Phi_\Q(6)$. In particular this is the first known example of a sporadic noncyclic torsion group (over degree $6$). See Remark \ref{spo} for more information about this sporadic torsion.

\section{Computations}
Let $E/\Q$ be an elliptic curve and $K/\Q$ a number field. We say that the torsion of $E$ grows from $\Q$ to $K$ if $E(\Q)_{\tors}\subsetneq E(K)_{\tors}$. Note that if the torsion of $E$  grows from $\Q$ to $K$, then of course the torsion of $E$ also grows from $\Q$ to any extension of $K$. We say that the torsion growth from $\Q$ to $K$ is primitive if $E(K')_{\tors}\subsetneq E(K)_{\tors}$ for any subfield $K'\subsetneq K$. With this definition it is possible to give a more detailed description of how the torsion grows in extensions of $\Q$. Given an elliptic curve $E/\Q$ and a positive integer $d$, we denote by $\mathcal{H}_{\Q}(d,E)$ the set of pairs $(H,K)$ (up to isomorphisms) where $K/\Q$ is an extension of degree dividing $d$, $H=E(K)_{\tors}\ne E(\Q)_{\tors}$ and the torsion growth in $K$ is primitive. Note that we are allowing the possibility of two (or more) of the torsion subgroups $H$ being isomorphic if the corresponding number fields $K$ are not isomorphic. We call the set $\mathcal{H}_{\Q}(d,E)$ the set of torsion configurations of the elliptic curve $E/\Q$. We let $\mathcal{H}_{\Q}(d)$ denote the set of $\mathcal{H}_{\Q}(d,E)$ as $E/\Q$ runs over all elliptic curves. Finally, for any $G\in \Phi(1)$ we define $\mathcal{H}_{\Q}(d,G)$ as the set of sets $\mathcal{H}_{\Q}(d,E)$ where $E/\Q$ runs over all the elliptic curve such that $E(\Q)_{\tors}=G$. 

Note that if we denote the maximum of the cardinality of the sets $S$ when $S\in \mathcal{H}_{\Q}(d)$ by $h_\Q(d)$, then $h_\Q(d)$ gives the maximum number of primitive degree $d$ extensions that can appear.

The sets $\mathcal{H}_{\Q}(d)$ and $\mathcal{H}_{\Q}(d,G)$, for any $G\in\Phi(1)$, have been completely determined for $d=2$ in \cite{GJT15}, for $d=3$ in \cite{GJNT16}, for $d=5$ in \cite{GJ17}, for $d=7$ and for any $d$ not divisible by a prime less than $11$ in \cite{GJN16}. For the case $d=4$, in order to conjecture what $\mathcal{H}_{\Q}(4)$ may be, $\mathcal{H}_\Q(4,E)$ has been computed for all elliptic curves over $\Q$ with conductor less than $350000$. In fact,\footnote{Filip Najman and the second author has developed \cite{GJN18} a (fast) algorithm that takes as input an elliptic curve defined over $\Q$ and an integer $d$ and returns the torsion configurations of degree $d$ for $E$, that is, $\mathcal{H}_\Q(d,E)$. At this moment the algorithm has been run on all elliptic curves Cremona's database and the torsion growth appear in the LMFDB webpage \cite{lmfdb} until degree $7$.} this computations has been enlarged to include elliptic curve with conductor up to $400000$.

In the same vein as the quartic computations we have carried out similar computations for the case when $d=6$. Table \ref{tablagrande} gives all\footnote{Including the elliptic curves \texttt{162d1} and \texttt{1296h1} that give sporadic torsion $(4,12)$ over sextic fields (see Section \ref{sporadic}).} the torsion configurations over sextic fields that we have found. We found $137$ torsion configurations in total and note all of the found configurations occur for an elliptic curve with conductor less than or equal to $10816$, far from $400000$. 

The following table shows what is know about $h_{\Q}(d) $:
$$
\begin{array}{|c|c|c|c|c|c|c|c|}
\hline
d & 2 & 3 & 4 & 5 & 6 & 7 & 2,3,5,7\nmid d \\
\hline
h_{\Q}(d) & 4 & 3 & \ge 9& 1 &\ge 9 & 1 & 0\\
\hline  
\end{array}
$$

\newpage

\section*{Appendix: Images of Mod $p$ Galois representations associated to elliptic curves over $\Q$}
For each possible subgroup $G_E(p)\subseteq \GL_2(\Z/p\Z)$ for $p\in R_\Q(6)=\{2, 3, 5, 7,13\}$, Table \ref{TABLE} lists in the first and second column the corresponding labels in Sutherland and Zywina notations, and the following data: 
\begin{itemize}
\item[$d_0$:] the index of the largest subgroup of $G_E(p)$ that fixes a $\Z/p\Z$-submodule of rank $1$ of $E[p]$; equivalently, the degree of the minimal extension $L/\Q$ over which $E$ admits a $p$-isogeny defined over $L$.
\item[$d_v$:] is the index of the stabilizers of $v\in(\Z/p\Z)^2$, $v\ne (0,0)$, by the action of $G_E(p)$ on $(\Z/p\Z)^2$; equivalently, the degrees of the extension $L/\Q$ over which $E$ has a $L$-rational point of order $p$.
\item[$d$:] is the order of $G_E(p)$; equivalently, the degree of the minimal extension $L/\Q$ for which $E[p]\subseteq E(L)$.
\end{itemize}
Note that Table \ref{TABLE} is partially extracted from Table 3 of \cite{Sutherland2}. The difference is that \cite[Table 3]{Sutherland2} only lists the minimum of $d_v$, which is denoted by $d_1$ therein.

\begin{table}[hb]
\begin{footnotesize}
\begin{tabular}{cc}
\begin{tabular}{@{\hskip -6pt}cllccc}
& \multicolumn{1}{c}{Sutherland}  &  \multicolumn{1}{c}{Zywina}  &$d_0$&$d_v $ & $d$\\\toprule
&\texttt{2Cs} & $G_1$  & 1 & 1 & 1\\ 
&\texttt{2B} &$G_2$ &1 & 1\,,\,2 & 2\\ 
&\texttt{2Cn} & $G_3$ & 3 & 3  & 3\\
 \multicolumn{3}{c}{$\,\,\,\,\,\,\GL_2(\Z/2\Z)$} & 3& 3& 6\\
\midrule
&\texttt{3Cs.1.1} &$H_{1,1}$ & 1& 1\,,\,2 & 2\\
&\texttt{3Cs} & $G_1$ &  1 & 2\,,\,4  & 4 \\
 &\texttt{3B.1.1} & $H_{3,1}$& 1 & 1\,,\,6  &6 \\
&\texttt{3B.1.2} &  $H_{3,2}$& 1 & 2\,,\,3  & 6\\
&\texttt{3Ns} &$G_2$ & 2 & 4  & 8\\
&\texttt{3B} & $G_3$ & 1 & 2\,,\,6  & 12\\
&\texttt{3Nn} & $G_4$ & 4 & 8  & 16\\
 \multicolumn{3}{c}{$\,\,\,\,\,\,\GL_2(\Z/3\Z)$} & 4& 8& 48\\
\midrule
&\texttt{5Cs.1.1} & $H_{1,1}$ & 1 & 1\,,\,4  & 4\\
&\texttt{5Cs.1.3} & $H_{1,2}$& 1 & 2\,,\,4  & 4\\
&\texttt{5Cs.4.1} & $G_1$ & 1 & 2\,,\,4\,,\,8  & 8\\
&\texttt{5Ns.2.1} & $G_3$ & 2 & 8\,,\,16  & 16\\
&\texttt{5Cs} &  $G_2$ & 1 & 4  & 16\\
&\texttt{5B.1.1} & $H_{6,1}$ & 1 & 1\,,\,20  & 20\\
&\texttt{5B.1.2} & $H_{5,1}$ & 1 & 4\,,\,5  & 20\\
&\texttt{5B.1.4} & $H_{6,2}$ & 1 & 2\,,\,20  & 20\\
&\texttt{5B.1.3} & $H_{5,2}$ & 1 & 4\,,\,10 & 20\\
&\texttt{5Ns} & $G_{4}$ & 2 & 8\,,\,16  & 32\\
&\texttt{5B.4.1} & $G_{6}$ & 1 & 2\,,\,20  & 40\\
&\texttt{5B.4.2} & $G_{5}$ & 1 & 4\,,\,10  & 40\\
&\texttt{5Nn} & $G_{7}$ & 6 & 24  & 48\\
&\texttt{5B} & $G_{8}$ & 1 & 4\,,\,20 & 80\\
&\texttt{5S4} & $G_{9}$ & 6 & 24  & 96 \\
  \multicolumn{3}{c}{$\,\,\,\,\,\,\GL_2(\Z/5\Z)$}  & 6& 24& 480\\
\bottomrule

\end{tabular}

&\qquad\quad

\begin{tabular}{@{\hskip -6pt}cllccc}
& \multicolumn{1}{c}{Sutherland}  &  \multicolumn{1}{c}{Zywina}   &$d_0$&$d_v $ & $d$\\\toprule
&\texttt{7Ns.2.1} &  $H_{1,1}$ &   $ 2$ & $ 6\,,\,9\,,\,18$  & 18 \\
&\texttt{7Ns.3.1} & $G_{1}$ &  $ 2$ & $12\,,\,18$ & 36 \\
&\texttt{7B.1.1} & $H_{3,1}$ &  $ 1$ & $ 1\,,\,42$ & 42 \\
&\texttt{7B.1.3} & $H_{4,1}$ &  $1 $ & $ 6\,,\, 7$  & 42 \\
 &\texttt{7B.1.2} & $H_{5,2}$ &  $ 1$ & $ 3\,,\,42$ & 42 \\
&\texttt{7B.1.5} & $H_{5,1}$ &   $ 1$ & $ 6\,,\,21$   & 42 \\
 &\texttt{7B.1.6} & $H_{3,2}$ &   $ 1$ & $ 2\,,\,21$ & 42 \\
 &\texttt{7B.1.4} &$H_{4,2}$ &   $ 1$ & $ 3\,,\,14$   & 42 \\
 &\texttt{7Ns} & $G_{2}$ &  $ 2$ &  $ 12\,,\,36$ & 72 \\
  &\texttt{7B.6.1} & $G_{3}$  &  $1 $ & $ 2\,,\,42$   & 84 \\
  &\texttt{7B.6.3} & $G_{4}$  &  $ 1$ & $ 6\,,\, 14$   & 84 \\
  &\texttt{7B.6.2} & $G_{5}$  &  $ 1$ & $ 6\,,\, 42$   & 84 \\
  &\texttt{7Nn} & $G_{6}$  &  $ 8$ & $ 48$&    96 \\
   &\texttt{7B.2.1} & $H_{7,2}$&   $1 $ & $ 3\,,\,42$   & 126 \\
   &\texttt{7B.2.3} & $H_{7,1}$ &   $ 1$ & $ 6\,,\, 21$   & 126 \\
   &\texttt{7B} & $G_{7}$ &  $1 $ & $ 6\,,\, 42$ & 252 \\
   \multicolumn{3}{c}{$\,\,\,\,\,\,\GL_2(\Z/7\Z)$  }&8& 48& 2016\\

\midrule
&\texttt{13S4} &$G_{7}$ &  $ 6$ & $72\,,\,96$ & 288 \\
 &\texttt{13B.3.1} &$H_{5,1}$ & $1 $ &  $3 \,,\, 156$ & 468 \\
 &\texttt{13B.3.2} & $H_{4,1}$ &  $1 $ & $ 12\,,\, 39$ & 468 \\
 &\texttt{13B.3.4} & $H_{5,2}$ &  $ 1$ & $ 6\,,\,156$ & 468 \\
 &\texttt{13B.3.7} & $H_{4,2}$ &  $ 1$ & $ 12 \,,\, 78$ & 468 \\
  &\texttt{13B.5.1} & $G_{2}$ &  $ 1$ & $ 4 \,,\, 156$ & 624 \\
  &\texttt{13B.5.2} & $G_{1}$ &  $ 1$ & $ 12 \,,\, 52$ & 624\\
  &\texttt{13B.5.4} & $G_{3}$ &  $ 1$ & $ 12 \,,\, 156$ & 624 \\
  &\texttt{13B.4.1} & $G_{5}$ &  $ 1$ & $ 6 \,,\, 156$  & 936 \\
  &\texttt{13B.4.2} & $G_{4}$ &  $ 1$ & $ 12 \,,\, 78$   & 936 \\
  &\texttt{13B} & $G_{6}$ &  $ 1$ & $ 12 \,,\, 156$  & 1872 \\
     \multicolumn{3}{c}{$\,\,\,\,\,\,\GL_2(\Z/13\Z)$}  &14& 168 & 26208 \\
 \bottomrule
\end{tabular}
\bigskip
\smallskip
\end{tabular}
\end{footnotesize}
\caption{Image groups $G_E(p)=\rho_{E,p}(\Gal(\Qbar/\Q))$, for $p\le 13$, for non-CM elliptic curves $E/\Q$.}
\label{TABLE}
\end{table}

\begin{table}[h]
\caption{\small Examples of elliptic curves such that $G\in\Phi(1)$ and $H\in\Phi_\Q(6,G)$ but $H\notin \Phi_\Q(d,G)$, $d=2,3$}\label{ex_6}
\begin{tabular}{|c|c|c|c|}
\hline
$G$ & $H$ & Sextic $K$ & Label of $E/\Q$ \\
\hline
\multirow{16}{*}{$(1)$} & $( 10 )$ & $x^6 + 2x^5 + 2x - 1$ & \href{http://www.lmfdb.org/EllipticCurve/Q/50a4}{\texttt{50a4}}\\ 
  & $( 12 )$ & $x^6 - 3x^4 + 2x^3 + 9x^2 - 12x + 4$ & \href{http://www.lmfdb.org/EllipticCurve/Q/162a2}{\texttt{162a2}}\\ 
  & $( 14 )$ & $x^6 + 2x^5 + 8x^4 + 8x^3 + 14x^2 + 4x + 4$ & \href{http://www.lmfdb.org/EllipticCurve/Q/208d1}{\texttt{208d1}}\\ 
  & $( 15 )$ & $x^6 - 5$ & \href{http://www.lmfdb.org/EllipticCurve/Q/50a4}{\texttt{50a4}}\\ 
  & $( 18 )$ & $x^6 - 3x^5 + 3x^3 + 6x^2 - 9x + 3$ & \href{http://www.lmfdb.org/EllipticCurve/Q/54a2}{\texttt{54a2}}\\ 
  & $( 21 )$ & $x^6 + x^3 + 1$ & \href{http://www.lmfdb.org/EllipticCurve/Q/162b2}{\texttt{162b2}}\\ 
  & $( 2, 6 )$ & $x^6 - 3x^5 + 5x^3 - 3x + 1$ & \href{http://www.lmfdb.org/EllipticCurve/Q/27a2}{\texttt{27a2}}\\ 
  & $( 2, 10 )$ & $x^6 - x^5 + 2x^4 - 3x^3 + 2x^2 - x + 1$ & \href{http://www.lmfdb.org/EllipticCurve/Q/121d1}{\texttt{121d1}}\\ 
 & $( 2, 18 )$ & $x^6 - 3x^2 + 6$& \href{http://www.lmfdb.org/EllipticCurve/Q/1728e3}{\texttt{1728e3}}\\ 
 & $( 3, 3 )$ & $x^6 + 3x^5 + 4x^4 + 3x^3 - 5x^2 - 6x + 12$ & \href{http://www.lmfdb.org/EllipticCurve/Q/19a2}{\texttt{19a2}}\\ 
 & $( 3, 9 )$ & $x^6 + 3$ & \href{http://www.lmfdb.org/EllipticCurve/Q/54a2}{\texttt{54a2}}\\ 
 & $( 4, 4 )$ & $x^6 - 2x^3 + 9x^2 + 6x + 2$ & \href{http://www.lmfdb.org/EllipticCurve/Q/162d2}{\texttt{162d2}}\\ 
 & \cellcolor{light-gray} $(4,12)$ &\cellcolor{light-gray}  $x^6 + 2x^3 + 9x^2 - 6x + 2$ &\cellcolor{light-gray}  \texttt{1296h1}\\ 
 & $( 6, 6 )$ & $x^6 + 3x^5 - 5x^3 + 3x + 1$ & \href{http://www.lmfdb.org/EllipticCurve/Q/108a2}{\texttt{108a2}}\\ 
\hline
\multirow{6}{*}{$(2)$}    & $( 18 )$ & $x^6 + 4x^3 + 7$ & \href{http://www.lmfdb.org/EllipticCurve/Q/14a3}{\texttt{14a3}}\\ 
  & $( 2, 14 )$ & $x^6 - x^5 + x^4 - x^3 + x^2 - x + 1$ & \href{http://www.lmfdb.org/EllipticCurve/Q/49a1}{\texttt{49a1}}\\ 
 & $( 2, 18 )$ & $x^6 - x^5 + x^4 - x^3 + x^2 - x + 1$ & \href{http://www.lmfdb.org/EllipticCurve/Q/98a1}{\texttt{98a1}}\\ 
  & $( 3, 6 )$ & $x^6 + 3x^5 - 3x^4 - 4x^3 + 69x^2 + 201x + 181$ & \href{http://www.lmfdb.org/EllipticCurve/Q/14a3}{\texttt{14a3}}\\ 
  & $( 3, 12 ) $& $x^6 + 3 $& \href{http://www.lmfdb.org/EllipticCurve/Q/30a3}{\texttt{30a3}}\\ 
  & $( 6, 6 )$ & $x^6 + 3x^5 - 5x^3 + 3x + 1$ & \href{http://www.lmfdb.org/EllipticCurve/Q/36a3}{\texttt{36a3}}\\ 
 \hline
\multirow{6}{*}{$(3)$} & $( 30 )$ & $x^6 - 2x^5 - 2x - 1$ & \href{http://www.lmfdb.org/EllipticCurve/Q/50a3}{\texttt{50a3}}\\ 
 & $( 3, 6 )$ & $x^6 + x^5 + 8x^4 + 5x^3 + 10x^2 + 3x + 3$ & \href{http://www.lmfdb.org/EllipticCurve/Q/19a1}{\texttt{19a1}}\\ 
 & $( 3, 9 )$ & $x^6 - x^3 + 1$ & \href{http://www.lmfdb.org/EllipticCurve/Q/27a1}{\texttt{27a1}}\\ 
 &\cellcolor{light-gray} $( 4, 12 )$ & \cellcolor{light-gray}$x^6 + 9x^4 - 12x^2 + 4$ &\cellcolor{light-gray}  \href{http://www.lmfdb.org/EllipticCurve/Q/162d1}{\texttt{162d1}}\\ 
 & $( 6, 6 )$ & $x^6 - 3x^5 + 6x^4 - 9x^3 + 12x^2 - 9x + 3$ & \href{http://www.lmfdb.org/EllipticCurve/Q/27a1}{\texttt{27a1}}\\ 
\hline
 $( 4 )$ &$ ( 3, 12 )$ & $x^6-3x^5+4x^4-3x^3-2x^2+3x+3$ & \href{http://www.lmfdb.org/EllipticCurve/Q/90c1}{\texttt{90c1}}\\ 
\hline
\multirow{2}{*}{$(5)$} & $( 30 )$ & $x^6 - 2x^5 - 2x - 1$ & \href{http://www.lmfdb.org/EllipticCurve/Q/50b1}{\texttt{50b1}}\\ 
 & $( 2, 10 )$ & $x^6 + x^5 + 2x^4 + 3x^3 + 2x^2 + x + 1$ & \href{http://www.lmfdb.org/EllipticCurve/Q/11a1}{\texttt{11a1}}\\ 
\hline
\multirow{3}{*}{$(6)$} & $( 2, 18 )$ & $x^6 - x^5 + x^4 - x^3 + x^2 - x + 1$ & \href{http://www.lmfdb.org/EllipticCurve/Q/14a4}{\texttt{14a4}}\\ 
 & $( 3, 12 )$ & $x^6 - 3x^5 + 4x^4 - 3x^3 - 2x^2 + 3x + 3$ &\href{http://www.lmfdb.org/EllipticCurve/Q/30a1}{\texttt{30a1}}\\ 
 & $( 6, 6 )$ & $x^6 - 3x^5 + 6x^4 - 9x^3 + 12x^2 - 9x + 3$ & \href{http://www.lmfdb.org/EllipticCurve/Q/36a1}{\texttt{36a1}}\\ 
\hline
$ ( 7 )$ & $( 2, 14 )$ & $x^6 - 2x^5 + 5x^4 + 4x^3 + 22x^2 + 16x + 16$ & \href{http://www.lmfdb.org/EllipticCurve/Q/26b1}{\texttt{26b1}}\\ 
\hline
\multirow{2}{*}{$(9)$} & $( 2, 18 )$ & $x^6 - 3x^2 + 6$ & \href{http://www.lmfdb.org/EllipticCurve/Q/54b3}{\texttt{54b3}}\\ 
 & $( 3, 9 )$ & $x^6 + 3$ & \href{http://www.lmfdb.org/EllipticCurve/Q/54b3}{\texttt{54b3}}\\ 
\hline
$ ( 12 ) $& $( 3, 12 )$ & $x^6 + 3$ & \href{http://www.lmfdb.org/EllipticCurve/Q/90c3}{\texttt{90c3}}\\ 
\hline
$ ( 2, 2 )$ & $( 6, 6 )$ & $x^6 + 3$ & \href{http://www.lmfdb.org/EllipticCurve/Q/30a6}{\texttt{30a6}}\\ 
\hline
$ ( 2, 6 )$ & $( 6, 6 )$ & $x^6 + 3x^5 + 4x^4 + 3x^3 - 2x^2 - 3x + 3$ & \href{http://www.lmfdb.org/EllipticCurve/Q/30a2}{\texttt{30a2}}\\
 \hline
\end{tabular}
\end{table}

\newpage

{\footnotesize
\scriptsize
\renewcommand{\arraystretch}{1.2}
\begin{table}[ht]
\caption{Torsion configurations over sextic fields}\label{tablagrande}
\begin{tabular}{|c|l|c|}
\hline
$G$ & $\mathcal{H}_\Q(6,E)$ & Label\\
\hline
\multirow{54}{*}{$(1)$} & $(2,2)$ &  \texttt{392b1}\\  \cline{2-3}  
 & $ (2,14)$ &  \href{http://www.lmfdb.org/EllipticCurve/Q/1922c1}{\texttt{1922c1}}\\  \cline{2-3}  
 & $ (2),(2,2)$ &  \href{http://www.lmfdb.org/EllipticCurve/Q/11a2}{\texttt{11a2}}\\  \cline{2-3}  
 & $ (2,2),(2,14)$ &  \href{http://www.lmfdb.org/EllipticCurve/Q/1922e1}{\texttt{1922e1}}\\  \cline{2-3}  
 & $ (4),(4,4)$ &  \href{http://www.lmfdb.org/EllipticCurve/Q/648a1}{\texttt{648a1}}\\  \cline{2-3}  
 & $ (7),(2,2)$ &  \href{http://www.lmfdb.org/EllipticCurve/Q/1922c2}{\texttt{1922c2}}\\  \cline{2-3}  
 & $ (2),(5),(2,10)$ &  \href{http://www.lmfdb.org/EllipticCurve/Q/121d1}{\texttt{121d1}}\\  \cline{2-3}  
 & $ (2),(7),(2,2)$ &  \href{http://www.lmfdb.org/EllipticCurve/Q/26b2}{\texttt{26b2}}\\  \cline{2-3}  
 & $ (2),(7),(2,14)$ &  \href{http://www.lmfdb.org/EllipticCurve/Q/10816bk1}{\texttt{10816bk1}}\\  \cline{2-3}  
 & $ (2),(13),(2,2)$ &  \href{http://www.lmfdb.org/EllipticCurve/Q/147c1}{\texttt{147c1}}\\  \cline{2-3}  
 & $ (3),(6),(6,6)$ &  \href{http://www.lmfdb.org/EllipticCurve/Q/108a2}{\texttt{108a2}}\\  \cline{2-3}  
 & $ (4),(7),(4,4)$ &  \href{http://www.lmfdb.org/EllipticCurve/Q/338b1}{\texttt{338b1}}\\  \cline{2-3}  
 & $ (2),(3)^2,(2,6)$ &  \href{http://www.lmfdb.org/EllipticCurve/Q/484a1}{\texttt{484a1}}\\  \cline{2-3}  
 & $ (2),(3),(9),(2,18)$ &  \href{http://www.lmfdb.org/EllipticCurve/Q/1728e3}{\texttt{1728e3}}\\  \cline{2-3}  
 & $ (2),(4)^2,(2,2)$ &  \href{http://www.lmfdb.org/EllipticCurve/Q/648c1}{\texttt{648c1}}\\  \cline{2-3}  
 & $ (2),(5),(10),(2,2)$ &  \href{http://www.lmfdb.org/EllipticCurve/Q/75a2}{\texttt{75a2}}\\  \cline{2-3}  
 & $ (2),(7),(14),(2,2)$ &  \href{http://www.lmfdb.org/EllipticCurve/Q/208d1}{\texttt{208d1}}\\  \cline{2-3}  
 & $ (3)^2,(2,2),(2,6)$ &  \href{http://www.lmfdb.org/EllipticCurve/Q/196a1}{\texttt{196a1}}\\  \cline{2-3}  
 & \cellcolor{light-gray}   $ (3)^2,(4),(4,12)$ & \cellcolor{light-gray}  \href{http://www.lmfdb.org/EllipticCurve/Q/1296h1}{\texttt{1296h1}}\\  \cline{2-3}  
 & $ (2),(3)^2,(2,6),(3,3)$ &  \href{http://www.lmfdb.org/EllipticCurve/Q/225b2}{\texttt{225b2}}\\  \cline{2-3}  
 & $ (2),(3)^2,(6),(2,2)$ &  \href{http://www.lmfdb.org/EllipticCurve/Q/50b3}{\texttt{50b3}}\\  \cline{2-3}  
 & $ (2),(3)^2,(6),(2,6)$ &  \href{http://www.lmfdb.org/EllipticCurve/Q/361b2}{\texttt{361b2}}\\  \cline{2-3}  
 & $ (2),(3)^2,(7),(2,6)$ &  \href{http://www.lmfdb.org/EllipticCurve/Q/5184bd1}{\texttt{5184bd1}}\\  \cline{2-3}  
 & $ (2),(3)^2,(9),(2,6)$ &  \href{http://www.lmfdb.org/EllipticCurve/Q/361b1}{\texttt{361b1}}\\  \cline{2-3}  
 & $ (2),(3)^2,(21),(2,6)$ &  \href{http://www.lmfdb.org/EllipticCurve/Q/5184b1}{\texttt{5184u1}}\\  \cline{2-3}  
 & $ (2),(3),(6)^2,(2,2)$ &  \href{http://www.lmfdb.org/EllipticCurve/Q/300a1}{\texttt{300a1}}\\  \cline{2-3}  
 & $ (2),(3),(9),(18),(2,2)$ &  \href{http://www.lmfdb.org/EllipticCurve/Q/432e3}{\texttt{432e3}}\\  \cline{2-3}  
 & $ (2),(4)^2,(7),(2,2)$ &  \href{http://www.lmfdb.org/EllipticCurve/Q/338a1}{\texttt{338a1}}\\  \cline{2-3}  
 & $ (3)^2,(2,2),(2,6),(3,3)$ &  \href{http://www.lmfdb.org/EllipticCurve/Q/196b2}{\texttt{196b2}}\\  \cline{2-3}  
 & $ (3)^2,(4),(12),(4,4)$ &  \href{http://www.lmfdb.org/EllipticCurve/Q/1296h2}{\texttt{1296h2}}\\  \cline{2-3}  
 & $ (2),(3)^2,(4)^2,(2,6)$ &  \href{http://www.lmfdb.org/EllipticCurve/Q/1296j1}{\texttt{1296j1}}\\  \cline{2-3}  
 & $ (2),(3)^2,(4),(12),(2,2)$ &  \href{http://www.lmfdb.org/EllipticCurve/Q/1296j2}{\texttt{1296j2}}\\  \cline{2-3}  
 & $ (2),(3)^2,(5),(6),(2,10)$ &  \href{http://www.lmfdb.org/EllipticCurve/Q/1600q1}{\texttt{1600q1}}\\  \cline{2-3}  
 & $ (2),(3)^2,(5),(10),(2,6)$ &  \href{http://www.lmfdb.org/EllipticCurve/Q/1600v3}{\texttt{1600v3}}\\  \cline{2-3}  
 & $ (2),(3)^2,(6),(2,2),(3,3)$ &  \href{http://www.lmfdb.org/EllipticCurve/Q/44a2}{\texttt{44a2}}\\  \cline{2-3}  
 & $ (2),(3)^2,(6),(2,2),(3,9)$ &  \href{http://www.lmfdb.org/EllipticCurve/Q/486c2}{\texttt{486c2}}\\  \cline{2-3}  
 & $ (2),(3)^2,(6)^2,(2,2)$ &  \href{http://www.lmfdb.org/EllipticCurve/Q/175b2}{\texttt{175b2}}\\  \cline{2-3}  
 & $ (2),(3)^2,(6),(7),(2,2)$ &  \href{http://www.lmfdb.org/EllipticCurve/Q/1296e2}{\texttt{1296e2}}\\  \cline{2-3}  
 & $ (2),(3)^2,(6),(9),(2,2)$ &  \href{http://www.lmfdb.org/EllipticCurve/Q/175b1}{\texttt{175b1}}\\  \cline{2-3}  
 & $ (2),(3)^2,(6),(9),(2,6)$ &  \href{http://www.lmfdb.org/EllipticCurve/Q/1728e2}{\texttt{1728e2}}\\  \cline{2-3}  
 & $ (2),(3)^2,(6),(21),(2,2)$ &  \href{http://www.lmfdb.org/EllipticCurve/Q/1296e1}{\texttt{1296e1}}\\  \cline{2-3}  
 & $ (2),(3),(9),(18),(2,2),(3,9)$ &  \href{http://www.lmfdb.org/EllipticCurve/Q/54a2}{\texttt{54a2}}\\  \cline{2-3}  
 & $ (3)^2,(4),(12),(3,3),(4,4)$ &  \href{http://www.lmfdb.org/EllipticCurve/Q/162d2}{\texttt{162d2}}\\  \cline{2-3}  
 & $ (2),(3)^2,(4)^2,(6),(2,2)$ &  \href{http://www.lmfdb.org/EllipticCurve/Q/4050g2}{\texttt{4050g2}}\\  \cline{2-3}  
 & $ (2),(3)^2,(4),(12),(2,2),(3,3)$ &  \href{http://www.lmfdb.org/EllipticCurve/Q/162a2}{\texttt{162a2}}\\  \cline{2-3}  
 & $ (2),(3)^2,(5),(6),(10),(2,2)$ &  \href{http://www.lmfdb.org/EllipticCurve/Q/400b1}{\texttt{400b1}}\\  \cline{2-3}  
 & $ (2),(3)^2,(6)^2,(9),(2,2)$ &  \href{http://www.lmfdb.org/EllipticCurve/Q/432a1}{\texttt{432a1}}\\  \cline{2-3}  
 & $ (2),(3)^2,(6),(7),(2,2),(3,3)$ &  \href{http://www.lmfdb.org/EllipticCurve/Q/162b4}{\texttt{162b4}}\\  \cline{2-3}  
 & $ (2),(3)^2,(6),(7),(21),(2,2)$ &  \href{http://www.lmfdb.org/EllipticCurve/Q/7938u3}{\texttt{7938u3}}\\  \cline{2-3}  
 & $ (2),(3)^2,(6),(21),(2,2),(3,3)$ &  \href{http://www.lmfdb.org/EllipticCurve/Q/162c2}{\texttt{162c2}}\\  \cline{2-3}  
 & $ (2),(3)^2,(9)^2,(2,6),(3,3)$ &  \href{http://www.lmfdb.org/EllipticCurve/Q/27a2}{\texttt{27a2}}\\  \cline{2-3}  
 & $ (2),(3)^2,(6),(7),(21),(2,2),(3,3)$ &  \href{http://www.lmfdb.org/EllipticCurve/Q/162b2}{\texttt{162b2}}\\  \cline{2-3}  
 & $ (2),(3)^2,(6),(9)^2,(2,2),(3,3)$ & \href{http://www.lmfdb.org/EllipticCurve/Q/19a2}{\texttt{19a2}}\\  \cline{2-3}  
 & $ (2),(3)^2,(5),(6),(10),(15),(2,2),(3,3)$ &  \href{http://www.lmfdb.org/EllipticCurve/Q/50a4}{\texttt{50a4}}\\ \hline
 \end{tabular}
 \end{table}
}

{\footnotesize
\scriptsize
\renewcommand{\arraystretch}{1.2}
\begin{table}[ht]
\begin{tabular}{ccc}
\begin{tabular}{|c|l|c|}
\hline
$G$ & $\mathcal{H}_\Q(6,E)$ & Label\\
\hline
 \multirow{28}{*}{$(2)$} & $ (2,2)$ &   \href{http://www.lmfdb.org/EllipticCurve/Q/46a1}{\texttt{46a1}}\\  \cline{2-3}  
 & $ (2,10)$ &   \href{http://www.lmfdb.org/EllipticCurve/Q/450a3}{\texttt{450a3}}\\  \cline{2-3}  
 & $ (2,2),(2,14)$ &   \href{http://www.lmfdb.org/EllipticCurve/Q/49a1}{\texttt{49a1}}\\  \cline{2-3}  
 & $ (6),(2,6)$ &   \href{http://www.lmfdb.org/EllipticCurve/Q/80b1}{\texttt{80b1}}\\  \cline{2-3}  
 & $ (10),(2,2)$ &   \href{http://www.lmfdb.org/EllipticCurve/Q/150b3}{\texttt{150b3}}\\  \cline{2-3}  
 & $ (14),(2,2)$ &   \href{http://www.lmfdb.org/EllipticCurve/Q/49a2}{\texttt{49a2}}\\  \cline{2-3}  
 & $ (4)^2,(2,2)$ &   \href{http://www.lmfdb.org/EllipticCurve/Q/15a5}{\texttt{15a5}}\\  \cline{2-3}  
 & $ (4),(8),(2,2)$ &   \href{http://www.lmfdb.org/EllipticCurve/Q/24a6}{\texttt{24a6}}\\  \cline{2-3}  
 & $ (4),(16),(2,2)$ &   \href{http://www.lmfdb.org/EllipticCurve/Q/3150bk1}{\texttt{3150bk1}}\\  \cline{2-3}  
 & $ (6),(2,2),(2,6)$ &   \href{http://www.lmfdb.org/EllipticCurve/Q/80b3}{\texttt{80b3}}\\  \cline{2-3}  
 & $ (6),(2,6),(2,18)$ &   \href{http://www.lmfdb.org/EllipticCurve/Q/98a1}{\texttt{98a1}}\\  \cline{2-3}  
 & $ (6),(2,6),(6,6)$ &   \href{http://www.lmfdb.org/EllipticCurve/Q/36a3}{\texttt{36a3}}\\  \cline{2-3}  
 & $ (6)^2,(2,2)$ &   \href{http://www.lmfdb.org/EllipticCurve/Q/80b2}{\texttt{80b2}}\\  \cline{2-3}  
 & $ (8)^2,(2,2)$ &   \href{http://www.lmfdb.org/EllipticCurve/Q/2880bd6}{\texttt{2880bd6}}\\  \cline{2-3}  
 & $ (14),(2,2),(2,14)$ &   \href{http://www.lmfdb.org/EllipticCurve/Q/49a3}{\texttt{49a3}}\\  \cline{2-3}  
 & $ (4)^2,(6),(2,6)$ &   \href{http://www.lmfdb.org/EllipticCurve/Q/960o7}{\texttt{960o7}}\\  \cline{2-3}  
 & $ (4)^2,(12),(2,6)$ &   \href{http://www.lmfdb.org/EllipticCurve/Q/450g1}{\texttt{450g1}}\\  \cline{2-3}  
 & $ (4),(6),(12),(2,2)$ &   \href{http://www.lmfdb.org/EllipticCurve/Q/240b3}{\texttt{240b3}}\\  \cline{2-3}  
 & $ (4),(12),(2,2),(2,6)$ &   \href{http://www.lmfdb.org/EllipticCurve/Q/450g3}{\texttt{450g3}}\\  \cline{2-3}  
 & $ (6)^2,(18),(2,2)$ &   \href{http://www.lmfdb.org/EllipticCurve/Q/98a2}{\texttt{98a2}}\\  \cline{2-3}  
 & $ (6),(18),(2,2),(2,6)$ &   \href{http://www.lmfdb.org/EllipticCurve/Q/98a5}{\texttt{98a5}}\\  \cline{2-3}  
 & $ (4)^2,(6),(2,2),(2,6)$ &   \href{http://www.lmfdb.org/EllipticCurve/Q/960o4}{\texttt{960o4}}\\  \cline{2-3}  
 & $ (4)^2,(6)^2,(2,2)$ &   \href{http://www.lmfdb.org/EllipticCurve/Q/150c4}{\texttt{150c4}}\\  \cline{2-3}  
 & $ (4)^2,(6),(12),(2,2)$ &   \href{http://www.lmfdb.org/EllipticCurve/Q/240b1}{\texttt{240b1}}\\  \cline{2-3}  
 & $ (6)^2,(2,2),(2,6),(3,6)$ &   \href{http://www.lmfdb.org/EllipticCurve/Q/20a3}{\texttt{20a3}}\\  \cline{2-3}  
 & $ (4),(6),(12)^2,(2,2),(2,6),(3,12)$ &   \href{http://www.lmfdb.org/EllipticCurve/Q/30a3}{\texttt{30a3}}\\  \cline{2-3}  
 & $ (6)^2,(18)^2,(2,2),(2,6),(3,6)$ &   \href{http://www.lmfdb.org/EllipticCurve/Q/14a3}{\texttt{14a3}}\\  \cline{2-3}  
 & $ (4)^2,(6)^2,(12)^2,(2,2),(2,6),(3,6)$ &   \href{http://www.lmfdb.org/EllipticCurve/Q/30a7}{\texttt{30a7}}\\ \hline
 \cline{2-3}  
 \multirow{14}{*}{$(3)$} & $ (2,6),(3,3)$ &   \href{http://www.lmfdb.org/EllipticCurve/Q/196b1}{\texttt{196b1}}\\  \cline{2-3}  
 & $ (6),(6,6)$ &   \href{http://www.lmfdb.org/EllipticCurve/Q/108a1}{\texttt{108a1}}\\  \cline{2-3}  
 & $ (6),(2,6),(3,3)$ &   \href{http://www.lmfdb.org/EllipticCurve/Q/44a1}{\texttt{44a1}}\\  \cline{2-3}  
 & $  \cellcolor{light-gray} (12),(3,3),(4,12)$ & \cellcolor{light-gray}    \href{http://www.lmfdb.org/EllipticCurve/Q/162d1}{\texttt{162d1}}\\  \cline{2-3}  
 & $ (6),(2,6),(3,3),(3,6)$ &   \href{http://www.lmfdb.org/EllipticCurve/Q/19a1}{\texttt{19a1}}\\  \cline{2-3}  
 & $ (6),(3,3),(3,9),(6,6)$ &   \href{http://www.lmfdb.org/EllipticCurve/Q/27a1}{\texttt{27a1}}\\  \cline{2-3}  
 & $ (6),(9),(2,6),(3,9)$ &   \href{http://www.lmfdb.org/EllipticCurve/Q/486f1}{\texttt{486f1}}\\  \cline{2-3}  
 & $ (6),(21),(2,6),(3,3)$ &   \href{http://www.lmfdb.org/EllipticCurve/Q/162b1}{\texttt{162b1}}\\  \cline{2-3}  
 & $ (6),(2,6),(3,3),(3,6),(3,9)$ &   \href{http://www.lmfdb.org/EllipticCurve/Q/54a1}{\texttt{54a1}}\\  \cline{2-3}  
 & $ (6),(9),(3,3),(3,9),(6,6)$ &   \href{http://www.lmfdb.org/EllipticCurve/Q/27a3}{\texttt{27a3}}\\  \cline{2-3}  
 & $ (6),(9)^2,(2,6),(3,3)$ &   \href{http://www.lmfdb.org/EllipticCurve/Q/19a3}{\texttt{19a3}}\\  \cline{2-3}  
 & $ (6),(12)^2,(2,6),(3,3)$ &   \href{http://www.lmfdb.org/EllipticCurve/Q/162a1}{\texttt{162a1}}\\  \cline{2-3}  
 & $ (6),(15),(30),(2,6),(3,3)$ &   \href{http://www.lmfdb.org/EllipticCurve/Q/50a3}{\texttt{50a3}}\\  \cline{2-3}  
 & $ (6),(9),(2,6),(3,3),(3,6),(3,9)$ &   \href{http://www.lmfdb.org/EllipticCurve/Q/54b1}{\texttt{54b1}}\\
\hline

\end{tabular}
&
\begin{tabular}{|c|l|c|}
\hline
$G$ & $\mathcal{H}_\Q(6,E)$ & Label\\
\hline
\multirow{9}{*}{$(4)$}  & $(2,4)$ &   \href{http://www.lmfdb.org/EllipticCurve/Q/17a1}{\texttt{17a1}}\\  \cline{2-3}  
 & $ (2,8)$ &   \href{http://www.lmfdb.org/EllipticCurve/Q/192c6}{\texttt{192c6}}\\  \cline{2-3}  
 & $ (4,4)$ &   \href{http://www.lmfdb.org/EllipticCurve/Q/40a4}{\texttt{40a4}}\\  \cline{2-3}  
 & $ (12),(2,12)$ &   \href{http://www.lmfdb.org/EllipticCurve/Q/150c3}{\texttt{150c3}}\\  \cline{2-3}  
 & $ (8)^2,(2,4)$ &   \href{http://www.lmfdb.org/EllipticCurve/Q/15a7}{\texttt{15a7}}\\  \cline{2-3}  
 & $ (8)^2,(2,8)$ &   \href{http://www.lmfdb.org/EllipticCurve/Q/240d6}{\texttt{240d6}}\\  \cline{2-3}  
 & $ (12),(2,4),(2,12)$ &   \href{http://www.lmfdb.org/EllipticCurve/Q/150c1}{\texttt{150c1}}\\  \cline{2-3}  
 & $ (12)^2,(2,4)$ &   \href{http://www.lmfdb.org/EllipticCurve/Q/720j5}{\texttt{720j5}}\\  \cline{2-3}  
 & $ (12)^2,(2,4),(2,12),(3,12)$ &   \href{http://www.lmfdb.org/EllipticCurve/Q/90c1}{\texttt{90c1}}\\ \hline
\hline
\multirow{2}{*}{$(5)$}  & $(10),(2,10)$ &   \href{http://www.lmfdb.org/EllipticCurve/Q/11a1}{\texttt{11a1}}\\  \cline{2-3}
 & $(10),(15)^2,(30),(2,10)$ &   \href{http://www.lmfdb.org/EllipticCurve/Q/50b1}{\texttt{50b1}}\\ \hline

\multirow{5}{*}{$(6)$} & $(2,6),(3,6)$ &   \href{http://www.lmfdb.org/EllipticCurve/Q/14a1}{\texttt{14a1}}\\  \cline{2-3}  
 & $ (2,6),(6,6)$ &   \href{http://www.lmfdb.org/EllipticCurve/Q/36a1}{\texttt{36a1}}\\  \cline{2-3}  
 & $ (12)^2,(2,6),(3,6)$ &   \href{http://www.lmfdb.org/EllipticCurve/Q/30a4}{\texttt{30a4}}\\  \cline{2-3}  
 & $ (12)^2,(2,6),(3,12)$ &   \href{http://www.lmfdb.org/EllipticCurve/Q/30a1}{\texttt{30a1}}\\  \cline{2-3}  
 & $ (18)^2,(2,6),(2,18),(3,6)$ &  \href{http://www.lmfdb.org/EllipticCurve/Q/14a4}{\texttt{14a4}}\\ \hline
\hline
$(7)$ & $(14),(2,14)$ &   \href{http://www.lmfdb.org/EllipticCurve/Q/26b1}{\texttt{26b1}}\\ \hline
\hline
\multirow{2}{*}{$(8)$} & $(2,8)$ &   \href{http://www.lmfdb.org/EllipticCurve/Q/14a4}{\texttt{15a4}}\\  \cline{2-3}  
 & $(16)^2,(2,8)$ &   \href{http://www.lmfdb.org/EllipticCurve/Q/210e1}{\texttt{210e1}}\\ \hline
\hline
$(9)$ & $(18),(2,18),(3,9)$ &   \href{http://www.lmfdb.org/EllipticCurve/Q/54b3}{\texttt{54b3}}\\ \hline
\hline
$(10)$ & $(2,10)$ &   \href{http://www.lmfdb.org/EllipticCurve/Q/66c1}{\texttt{66c1}}\\ \hline
\hline
$(12)$ & $(2,12),(3,12)$ &   \href{http://www.lmfdb.org/EllipticCurve/Q/90c3}{\texttt{90c3}}\\ \hline
 \cline{2-3}  
\multirow{12}{*}{$(2,2)$} & $ (2,4)$ &   \href{http://www.lmfdb.org/EllipticCurve/Q/33a1}{\texttt{33a1}}\\  \cline{2-3}  
 & $ (2,8)$ &   \href{http://www.lmfdb.org/EllipticCurve/Q/63a2}{\texttt{63a2}}\\  \cline{2-3}  
 & $ (2,4)^2$ &   \href{http://www.lmfdb.org/EllipticCurve/Q/17a2}{\texttt{17a2}}\\  \cline{2-3}  
 & $ (2,4),(2,8)$ &   \href{http://www.lmfdb.org/EllipticCurve/Q/75b3}{\texttt{75b3}}\\  \cline{2-3}  
 & $ (2,6)^2$ &   \href{http://www.lmfdb.org/EllipticCurve/Q/240b2}{\texttt{240b2}}\\  \cline{2-3}  
 & $ (2,6),(2,12)$ &   \href{http://www.lmfdb.org/EllipticCurve/Q/960o6}{\texttt{960o6}}\\  \cline{2-3}  
 & $ (2,4)^3$ &   \href{http://www.lmfdb.org/EllipticCurve/Q/15a2}{\texttt{15a2}}\\  \cline{2-3}  
 & $ (2,4)^2,(2,8)$ &   \href{http://www.lmfdb.org/EllipticCurve/Q/510e5}{\texttt{510e5}}\\  \cline{2-3}  
 & $ (2,4),(2,6)^2$ &   \href{http://www.lmfdb.org/EllipticCurve/Q/150c2}{\texttt{150c2}}\\  \cline{2-3}  
 & $ (2,4),(2,6),(2,12)$ &   \href{http://www.lmfdb.org/EllipticCurve/Q/960o2}{\texttt{960o2}}\\  \cline{2-3}  
 & $ (2,6)^2,(6,6)$ &   \href{http://www.lmfdb.org/EllipticCurve/Q/30a6}{\texttt{30a6}}\\  \cline{2-3}  
 & $ (2,4),(2,6)^2,(2,12),(6,6)$ &   \href{http://www.lmfdb.org/EllipticCurve/Q/90c2}{\texttt{90c2}}\\ \hline
 \cline{2-3}  
\multirow{5}{*}{$(2,4)$} & $ (2,8)$ &   \href{http://www.lmfdb.org/EllipticCurve/Q/15a3}{\texttt{15a3}}\\  \cline{2-3}  
 & $ (4,4)$ &   \href{http://www.lmfdb.org/EllipticCurve/Q/195a3}{\texttt{195a3}}\\  \cline{2-3}  
 & $ (2,8)^2$ &   \href{http://www.lmfdb.org/EllipticCurve/Q/1230f2}{\texttt{1230f2}}\\  \cline{2-3}  
 & $ (2,8),(4,4)$ &   \href{http://www.lmfdb.org/EllipticCurve/Q/15a1}{\texttt{15a1}}\\  \cline{2-3}  
 & $ (2,8)^2,(4,4)$ &   \href{http://www.lmfdb.org/EllipticCurve/Q/210e3}{\texttt{210e3}}\\ \hline
 \cline{2-3}  
 \multirow{2}{*}{$(2,6)$} & $ (6,6)$ &   \href{http://www.lmfdb.org/EllipticCurve/Q/30a2}{\texttt{30a2}}\\  \cline{2-3}  
 & $ (2,12),(6,6)$ &   \href{http://www.lmfdb.org/EllipticCurve/Q/90c6}{\texttt{90c6}}\\ \hline
\end{tabular}
\end{tabular}
\end{table}
}

\end{document}